\documentclass{amsart}[11pt]
\usepackage{amsmath,amssymb,amsfonts,amsthm}
\usepackage{ifpdf}
\usepackage{enumerate}
\usepackage{ulem}
\usepackage{leftidx}
\usepackage{microtype}
\usepackage{afterpage}

\usepackage{breqn}

\usepackage{tikz}
\usetikzlibrary{calc}
\usepackage{tikz-qtree}

\usepackage[plainpages=false,hypertexnames=false,pdfpagelabels]{hyperref}

\newcommand{\arxiv}[1]{\href{http://arxiv.org/abs/#1}{\tt arXiv:\nolinkurl{#1}}}
\newcommand{\doi}[1]{\href{http://dx.doi.org/#1}{{\tt DOI:#1}}}

\newcommand{\googlebooks}[1]{(preview at \href{http://books.google.com/books?id=#1}{google books})}
\renewcommand{\googlebooks}[1]{}

\theoremstyle{plain}
\newtheorem{prop}{Proposition}[section]

\newtheorem{thm}[prop]{Theorem}
\newtheorem{thm*}{Theorem}

\newtheorem{fact}[prop]{Fact}

\newtheorem{lem}[prop]{Lemma}
\newtheorem{cor}[prop]{Corollary}
\newtheorem*{cor*}{Corollary}
\numberwithin{equation}{section}
\theoremstyle{remark}

\newtheorem{remark}[prop]{Remark}           
\newtheorem*{rem*}{Remark}               %unnumbered remark
\newtheorem*{ex*}{Example}                %unnumbered exercise

\theoremstyle{definition}
\newtheorem{defn}[prop]{Definition}         % numbered definition

\newtheorem*{defn*}{Definition}             % unnumbered definition

\theoremstyle{plain}

\newcounter{comment}
\newcommand{\noop}[1]{}

\def\clap#1{\hbox to 0pt{\hss#1\hss}}

\newcommand{\Integer}{\mathbb Z}

\newcommand{\Real}{\mathbb R}
\newcommand{\Complex}{\mathbb C}

\def\semicolon{;}
\def\applytolist#1{
    \expandafter\def\csname multi#1\endcsname##1{
        \def\multiack{##1}\ifx\multiack\semicolon
            \def\next{\relax}
        \else
            \csname #1\endcsname{##1}
            \def\next{\csname multi#1\endcsname}
        \fi
        \next}
    \csname multi#1\endcsname}

\def\calc#1{\expandafter\def\csname c#1\endcsname{{\mathcal #1}}}
\applytolist{calc}QWERTYUIOPLKJHGFDSAZXCVBNM;
\def\bbc#1{\expandafter\def\csname bb#1\endcsname{{\mathbb #1}}}
\applytolist{bbc}QWERTYUIOPLKJHGFDSAZXCVBNM;
\def\bfc#1{\expandafter\def\csname bf#1\endcsname{{\mathbf #1}}}
\applytolist{bfc}QWERTYUIOPLKJHGFDSAZXCVBNM;
\def\calc#1{\expandafter\def\csname f#1\endcsname{{\mathfrak #1}}}
\applytolist{calc}QWERTYUIOPLKJHGFDSAZXCVBNM;

\newcommand{\id}{\boldsymbol{1}}
\renewcommand{\imath}{\mathfrak{i}}
\renewcommand{\jmath}{\mathfrak{j}}

\newcommand{\iso}{\cong}

\newcommand{\set}[2]{\left\{#1\middle|#2\right\}}

\makeatletter
\newcommand{\hashdef}[2]{\@namedef{#1}{#2}}
\newcommand{\hashlookup}[1]{\@nameuse{#1}}
\makeatother

  \newcommand{\pathtographs}{diagrams/graphs/}
\hashdef{bwd1duals1}{FF1481245D7F1881}%
\hashdef{bwd1v1duals1v1}{7BFEC24E607C77CA}%
\hashdef{bwd1v1p1duals1v1x2}{7CC0C785EDD0C3AA}%
\hashdef{bwd1v1p1duals1v2x1}{FE6BDA913A50F4D9}%
\hashdef{bwd1v1p1p1duals1v1x2x3}{CDEBF7CA2995B855}%
\hashdef{bwd1v1p1p1duals1v2x1x3}{99E8BDA8F0BBDAF7}%
\hashdef{bwd1v1p1p1p1duals1v1x2x3x4}{5CDD3091D4E2F2AF}%
\hashdef{bwd1v1p1p1p1duals1v2x1x3x4}{72264B2B3ACD4595}%
\hashdef{bwd1v1p1p1p1duals1v2x1x4x3}{19385ECD51785D9D}%
\hashdef{bwd1v1p1p1p1p1duals1v1x2x3x4x5}{E99BE04176D69816}%
\hashdef{bwd1v1p1p1p1p1duals1v1x2x3x5x4}{7CE84C9B463B6309}%
\hashdef{bwd1v1p1p1p1p1duals1v1x3x2x5x4}{09011B72436D3D25}%
\hashdef{bwd1v1p1p1p1p1duals1v2x1x3x4x5}{71BC06FD9E0585EF}%
\hashdef{bwd1v1p1p1p1p1duals1v2x1x4x3x5}{24ADE260B724FA90}%
\hashdef{bwd1v1p1p1v0x0x1p0x0x1duals1v2x1x3}{3CA5711FC31ECF6F}%
\hashdef{bwd1v1p1p1v0x1x0p0x0x1duals1v1x2x3}{3817420434DA8793}%
\hashdef{bwd1v1p1p1v0x1x0p0x0x1duals1v1x3x2}{88D67A817E40ED00}%
\hashdef{bwd1v1p1p1v1x1x0duals1v1x2x3}{53D091157B8A128E}%
\hashdef{bwd1v1p1v0x1p1x1p0x1duals1v1x2}{BE3F0EFC4C101348}%
\hashdef{bwd1v1p1v0x1p1x1p0x1v0x1x0duals1v1x2v1}{9530D50FCB3A04BA}%
\hashdef{bwd1v1p1v0x1p1x1p0x1v1x0x0p0x0x1duals1v1x2v1x2}{C11020B9804C3017}%
\hashdef{bwd1v1p1v0x1p1x1p0x1v1x0x0p0x0x1v1x1duals1v1x2v1x2}{33ED0739F1699026}%
\hashdef{bwd1v1p1v0x1p1x1p0x1v1x0x0p0x0x1v1x1v1duals1v1x2v1x2v1}{74FE74D6E67ACD15}%
\hashdef{bwd1v1p1v0x1p1x1p0x1v1x0x0p0x1x0duals1v1x2v1x2}{12A73A94EE92B425}%
\hashdef{bwd1v1p1v0x1p1x1p0x1v1x0x0p0x1x0v1x0duals1v1x2v1x2}{1653A673B8B373AB}%
\hashdef{bwd1v1p1v0x1p1x1p0x1v1x0x0p0x1x0v1x0p1x0duals1v1x2v1x2}{5812FB9D2E43659F}%
\hashdef{bwd1v1p1v1x0p0x1v1x0p1x0p0x1v1x0x0p0x0x1v1x0p1x0p0x1p0x1v1x0x0x0p0x0x1x0duals1v1x2v1x3x2v1x3x2x4}{980EC644D3C98AB9}%
\hashdef{bwd1v1p1v1x0p0x1v1x0p1x0p0x1v1x0x0p0x0x1v1x0p1x0p1x0p0x1p0x1v0x1x0x1x0v1duals1v1x2v1x3x2v1x2x4x3x5v1}{E062CBA9D4D0BB85}%
\hashdef{bwd1v1p1v1x0p0x1v1x0p1x0p1x0p0x1v1x0x0x1v1duals1v1x2v1x2x4x3v1}{DBDF79CAF0076A8E}%
\hashdef{bwd1v1p1v1x0p1x0p1x1p0x1duals1v1x2}{1FFA335772596C5E}%
\hashdef{bwd1v1p1v1x0p1x0p1x1p0x1v0x0x0x1duals1v1x2v1}{FBEE5222FC6B59BD}%
\hashdef{bwd1v1p1v1x0p1x0p1x1p0x1v0x0x0x1p0x0x0x1duals1v1x2v1x2}{5355567C9CEADD71}%
\hashdef{bwd1v1p1v1x0p1x0p1x1p0x1v0x0x0x1p0x0x0x1duals1v1x2v2x1}{E5D9BA57C50AA3E0}%
\hashdef{bwd1v1p1v1x0p1x0p1x1p0x1v0x1x0x0p0x0x0x1duals1v1x2v1x2}{3E79B34E48DA5ED0}%
\hashdef{bwd1v1p1v1x0p1x0p1x1p0x1v1x0x0x0duals1v1x2v1}{5B9A2EB868BC55C2}%
\hashdef{bwd1v1p1v1x0p1x0p1x1p0x1v1x0x0x0p0x1x0x0duals1v1x2v1x2}{FB6A4A1051E958AE}%
\hashdef{bwd1v1p1v1x0p1x0p1x1p0x1v1x0x0x0p0x1x0x0duals1v1x2v2x1}{8E857E3485F9F024}%
\hashdef{bwd1v1p1v1x0p1x0v1x0p1x0v1x1duals1v1x2v1x2}{8F007BF50C3E09AC}%
\hashdef{bwd1v1p1v1x0p1x0v1x0v1p1v1x0p1x0v1x1duals1v1x2v1v1x2}{21FA8EFC6770814D}%
\hashdef{bwd1v1p1v1x0p1x0v1x0v1p1v1x0v1p1duals1v1x2v1v1}{D341886B743822B2}%
\hashdef{bwd1v1p1v1x0p1x1duals1v1x2}{FE060C2815A2CF0E}%
\hashdef{bwd1v1p1v1x0p1x1duals1v2x1}{842B8FC10A6878B8}%
\hashdef{bwd1v1p1v1x0p1x1p0x1duals1v1x2}{9509DE688F210AFD}%
\hashdef{bwd1v1p1v1x0p1x1p0x1duals1v2x1}{77A76213853D258A}%
\hashdef{bwd1v1p1v1x0p1x1p0x1p0x1duals1v1x2}{46DBCF7FD1143E2F}%
\hashdef{bwd1v1p1v1x0p1x1p0x1p0x1v1x0x0x0p0x0x1x0duals1v1x2v2x1}{CE1EAB4A62707FC0}%
\hashdef{bwd1v1p1v1x0p1x1p0x1v0x1x0duals1v1x2v1}{FF05E76A92AAD72C}%
\hashdef{bwd1v1p1v1x0p1x1p0x1v0x1x0duals1v2x1v1}{70D1C9A600E435C3}%
\hashdef{bwd1v1p1v1x0p1x1p0x1v0x1x0p0x0x1p0x0x1duals1v1x2v1x2x3}{E82A80DFF3721400}%
\hashdef{bwd1v1p1v1x0p1x1p0x1v0x1x0p0x0x1p0x0x1duals1v1x2v1x3x2}{393F67248EA27839}%
\hashdef{bwd1v1p1v1x0p1x1p0x1v0x1x0p0x0x1p0x0x1v0x0x1p0x0x1duals1v1x2v1x2x3}{F9939D95E6B5038E}%
\hashdef{bwd1v1p1v1x0p1x1p0x1v0x1x0p0x0x1p0x0x1v0x0x1p0x0x1v0x1p0x1duals1v1x2v1x2x3v1x2}{6D077B98D45881D3}%
\hashdef{bwd1v1p1v1x0p1x1p0x1v0x1x0p0x0x1p0x0x1v0x0x1p0x0x1v0x1p0x1duals1v1x2v1x2x3v2x1}{8F24887EF9DA8DD0}%
\hashdef{bwd1v1p1v1x0p1x1p0x1v0x1x0p0x0x1p0x0x1v0x0x1p0x0x1v1x0duals1v1x2v1x2x3v1}{14E4B1C69C87C568}%
\hashdef{bwd1v1p1v1x0p1x1p0x1v0x1x0p0x0x1p0x0x1v0x0x1p0x0x1v1x0p0x1duals1v1x2v1x2x3v1x2}{C7D9E1A9ED32502C}%
\hashdef{bwd1v1p1v1x0p1x1p0x1v0x1x0p0x0x1p0x0x1v0x0x1p0x0x1v1x0p1x0duals1v1x2v1x2x3v1x2}{32824FD30D0406DE}%
\hashdef{bwd1v1p1v1x0p1x1p0x1v0x1x0p0x0x1p0x0x1v0x0x1p0x0x1v1x0p1x0duals1v1x2v1x2x3v2x1}{29484856A6976AD5}%
\hashdef{bwd1v1p1v1x0p1x1p0x1v0x1x0p0x0x1p0x0x1v0x1x0p0x0x1duals1v1x2v1x2x3}{BF3B651A0F0EC7F8}%
\hashdef{bwd1v1p1v1x0p1x1p0x1v0x1x0p0x0x1p0x0x1v0x1x0p0x0x1p0x0x1duals1v1x2v1x2x3}{04995482E3C1668F}%
\hashdef{bwd1v1p1v1x0p1x1p0x1v0x1x0p0x0x1p0x0x1v0x1x0p0x0x1v1x0p0x1p0x1p0x1duals1v1x2v1x2x3v2x1x3x4}{334ACB14A0558A13}%
\hashdef{bwd1v1p1v1x0p1x1p0x1v0x1x0p0x0x1p0x0x1v0x1x0p0x0x1v1x0p1x0p0x1duals1v1x2v1x2x3v1x3x2}{F8F9652A68CA6486}%
\hashdef{bwd1v1p1v1x0p1x1p0x1v0x1x0p0x0x1p0x0x1v0x1x0p0x0x1v1x0p1x0p0x1p0x1duals1v1x2v1x2x3v1x3x2x4}{99BEC4C1CF252171}%
\hashdef{bwd1v1p1v1x0p1x1p0x1v0x1x0p0x0x1p0x0x1v0x1x0p0x0x1v1x0p1x0p1x0p0x1duals1v1x2v1x2x3v2x1x4x3}{CD34E12B5DC74B02}%
\hashdef{bwd1v1p1v1x0p1x1p0x1v0x1x0p0x0x1p0x0x1v0x1x0p0x0x1v1x1duals1v1x2v1x2x3v1}{B01E7E7CFF0ED2E6}%
\hashdef{bwd1v1p1v1x0p1x1p0x1v0x1x0p0x0x1p0x0x1v0x1x0p0x1x0p0x0x1duals1v1x2v1x2x3}{26BC0698B6AC68E6}%
\hashdef{bwd1v1p1v1x0p1x1p0x1v1x0x0p0x0x1duals1v1x2v1x2}{24E824D6786FCE3F}%
\hashdef{bwd1v1p1v1x0p1x1p0x1v1x0x0p0x0x1duals1v2x1v2x1}{C265D341C8761121}%
\hashdef{bwd1v1p1v1x0p1x1p0x1v1x0x0p0x1x0duals1v1x2v1x2}{A8FD432CDB07C847}%
\hashdef{bwd1v1p1v1x0p1x1p0x1v1x0x0p0x1x0p0x0x1duals1v1x2v1x2x3}{6BFD30B330C58653}%
\hashdef{bwd1v1p1v1x0p1x1p0x1v1x0x0p0x1x0p0x0x1duals1v1x2v3x2x1}{5A1FC012A9478C0F}%
\hashdef{bwd1v1p1v1x0p1x1p0x1v1x0x0p0x1x0p0x0x1duals1v2x1v3x2x1}{B521E1AB39E5C344}%
\hashdef{bwd1v1p1v1x0p1x1p0x1v1x0x0p0x1x0p0x0x1p0x0x1duals1v1x2v1x2x3x4}{EF2E422DA889E900}%
\hashdef{bwd1v1p1v1x0p1x1p0x1v1x0x0p0x1x0p0x0x1p0x0x1duals1v1x2v1x2x4x3}{E1B9923017B5773F}%
\hashdef{bwd1v1p1v1x0p1x1p0x1v1x0x0p1x0x0p0x1x0duals1v1x2v1x2x3}{CB2E56F99264124D}%
\hashdef{bwd1v1p1v1x0p1x1p0x1v1x0x0p1x0x0p0x1x0duals1v1x2v2x1x3}{E58E4DF4C4BA9F59}%
\hashdef{bwd1v1p1v1x0p1x1p0x1v1x0x0p1x0x0p0x1x0p0x0x1duals1v1x2v1x2x3x4}{F760A721E2EB0038}%
\hashdef{bwd1v1p1v1x0p1x1p0x1v1x0x0p1x0x0p0x1x0p0x0x1duals1v1x2v1x4x3x2}{981369E37BDFCA11}%
\hashdef{bwd1v1p1v1x0p1x1p0x1v1x0x0p1x0x0p0x1x0p0x0x1v0x0x0x1duals1v1x2v1x4x3x2}{CC1E8BCB63BDAEDE}%
\hashdef{bwd1v1p1v1x0p1x1p0x1v1x0x0p1x0x0p0x1x0p0x0x1v1x0x0x0p0x0x0x1duals1v1x2v1x4x3x2}{20B652F83AC42DEA}%
\hashdef{bwd1v1p1v1x0p1x1p0x1v1x0x0p1x0x0p0x1x0v1x0x0p0x1x0duals1v1x2v2x1x3}{D497FA627FB8AC90}%
\hashdef{bwd1v1p1v1x0p1x1p0x1v1x0x0p1x0x0p0x1x0v1x0x0p0x1x0v1x0p1x0p0x1p0x1duals1v1x2v2x1x3v1x3x2x4}{B84C224FF0EF50B9}%
\hashdef{bwd1v1p1v1x0p1x1p0x1v1x0x0p1x0x0p0x1x0v1x0x0p0x1x0v1x1duals1v1x2v2x1x3v1}{67D23E669480F467}%
\hashdef{bwd1v1p1v1x0p1x1p0x1v1x0x0p1x0x0p0x1x0v1x0x0p1x0x0duals1v1x2v1x2x3}{EC77A1C86B29B69B}%
\hashdef{bwd1v1p1v1x0p1x1p0x1v1x0x0p1x0x0p0x1x0v1x0x0p1x0x0v1x0p0x1duals1v1x2v1x2x3v2x1}{17573EE5A5580682}%
\hashdef{bwd1v1p1v1x0p1x1p0x1v1x0x1duals1v1x2v1}{733D2882101F804F}%
\hashdef{bwd1v1p1v1x0p1x1p0x1v1x0x1duals1v2x1v1}{1EC39430D365B163}%
\hashdef{bwd1v1p1v1x0p1x1p0x1v1x0x1v1duals1v1x2v1}{B8B6B89DAA5FD820}%
\hashdef{bwd1v1p1v1x0p1x1p0x1v1x0x1v1duals1v2x1v1}{09722ABEE1C287B1}%
\hashdef{bwd1v1p1v1x0p1x1p0x1v1x0x1v1v1p1duals1v1x2v1v1x2}{C899C5D1AB8C8DAF}%
\hashdef{bwd1v1p1v1x0p1x1p0x1v1x0x1v1v1p1duals1v1x2v1v2x1}{3EE21B5D10D969C4}%
\hashdef{bwd1v1p1v1x0p1x1p1x0duals1v1x2}{48156979AB520E1B}%
\hashdef{bwd1v1p1v1x0p1x1p1x0p0x1duals1v1x2}{EE219A0961E3A4BB}%
\hashdef{bwd1v1p1v1x0p1x1p1x0p0x1v0x0x1x0p0x0x0x1duals1v1x2v2x1}{379EDC5C6A656D89}%
\hashdef{bwd1v1p1v1x1duals1v1x2}{8BDF697267FBBFEB}%
\hashdef{bwd1v1p1v1x1duals1v2x1}{9AB6D35FB0C8481F}%
\hashdef{bwd1v1p1v1x1p0x1p0x1duals1v1x2}{28989488BFDC6E8F}%
\hashdef{bwd1v1p1v1x1p0x1p0x1v0x0x1duals1v1x2v1}{7DC7A8FA0AB66922}%
\hashdef{bwd1v1p1v1x1p0x1p0x1v0x1x0p0x0x1duals1v1x2v1x2}{638AAB78D215999E}%
\hashdef{bwd1v1p1v1x1p0x1p0x1v0x1x0p0x0x1duals1v1x2v2x1}{C612FA6A71EC47E1}%
\hashdef{bwd1v1p1v1x1p0x1p0x1v0x1x0p0x0x1p0x0x1duals1v1x2v1x3x2}{793432A20E92D6C0}%
\hashdef{bwd1v1p1v1x1p0x1p0x1v0x1x0p0x1x0duals1v1x2v1x2}{BD08B8B8A7C8C6E7}%
\hashdef{bwd1v1p1v1x1p0x1p0x1v0x1x0p0x1x0duals1v1x2v2x1}{636D5BAA4D54FFF1}%
\hashdef{bwd1v1p1v1x1p0x1p0x1v0x1x0p0x1x0p0x0x1duals1v1x2v1x2x3}{4404534C4AA7E003}%
\hashdef{bwd1v1p1v1x1p0x1p0x1v0x1x0p0x1x0p0x0x1duals1v1x2v1x3x2}{5046727BBF36BDA0}%
\hashdef{bwd1v1p1v1x1p0x1p0x1v0x1x0p0x1x0p0x0x1duals1v1x2v2x1x3}{30356E47F5705149}%
\hashdef{bwd1v1p1v1x1p0x1p0x1v0x1x0p0x1x0p0x0x1v0x0x1p1x0x0duals1v1x2v1x3x2}{8306D71AF166BFE4}%
\hashdef{bwd1v1p1v1x1p0x1p0x1v0x1x0p0x1x0p0x0x1v1x0x0p0x0x1duals1v1x2v1x3x2}{9586BE69412EA29E}%
\hashdef{bwd1v1p1v1x1p0x1p0x1v0x1x0p0x1x0v0x1p0x1duals1v1x2v1x2}{39614F78AB8D2199}%
\hashdef{bwd1v1p1v1x1p0x1p0x1v0x1x0p0x1x0v1x0p0x1duals1v1x2v1x2}{C393D07C61093079}%
\hashdef{bwd1v1p1v1x1p0x1p0x1v0x1x0p0x1x0v1x0p0x1duals1v1x2v2x1}{270C852579A737A4}%
\hashdef{bwd1v1p1v1x1p0x1p0x1v0x1x0p0x1x0v1x0p0x1p1x0duals1v1x2v1x2}{6C84259C9D924A06}%
\hashdef{bwd1v1p1v1x1p0x1p0x1v0x1x0p0x1x0v1x0p1x0duals1v1x2v1x2}{B816BB524EFCA0B5}%
\hashdef{bwd1v1p1v1x1p0x1p0x1v0x1x0p0x1x0v1x0p1x0p0x1duals1v1x2v1x2}{9ECD23B325923792}%
\hashdef{bwd1v1p1v1x1p0x1p0x1v1x0x0duals1v1x2v1}{D66843AFD40C9CA6}%
\hashdef{bwd1v1p1v1x1p0x1p0x1v1x0x0p0x1x0duals1v1x2v1x2}{F90255D8F87DFEB5}%
\hashdef{bwd1v1p1v1x1p0x1p1x0duals1v1x2}{E66520D22AAA98D8}%
\hashdef{bwd1v1p1v1x1p1x0p0x1duals1v1x2}{5B2FC066F9848D63}%
\hashdef{bwd1v1p1v1x1p1x0p1x0duals1v1x2}{48DF0E324CCE66A8}%
\hashdef{bwd1v1p1v1x1v1duals1v1x2v1}{434A46414835DC7C}%
\hashdef{bwd1v1p2duals1v1x2}{2C8DF9CB0DCAFA89}%
\hashdef{bwd1v1v1p1p1p1v1x0x0x0p0x1x0x0p0x0x1x0duals1v1v1x2x3}{544678A16CA85123}%
\hashdef{bwd1v1v1p1p1p1v1x0x0x0p0x1x0x0p0x0x1x0duals1v1v2x1x3}{35BFA8247EC53DC1}%
\hashdef{bwd1v2duals1v1}{FE04BE4B29C7BC79}%
\hashdef{bwd1v2p1duals1v1x2}{513FF62EE1033102}%
\hashdef{bwd1v2p1v0x1duals1v1x2}{3AA9762714DF9599}%
\hashdef{bwd1v2v1duals1v1}{328A9FB1D0CD1AD1}%
\hashdef{gbg1v1p1v1x0p1x1}{163B0A52CD5DB27D}%
\hashdef{gbg1v1p1v1x1}{9DF54046C143B43C}%

\newcommand{\bigraph}[1]{{\hspace{-3pt}\begin{array}{c}%
  \raisebox{-2.5pt}{\includegraphics[height=6mm]{\pathtographs \hashlookup{#1}}}% 
\end{array}\hspace{-3pt}}}

\newcommand{\smallbigraph}[1]{{\hspace{-3pt}\begin{array}{c}%
  \raisebox{-2.5pt}{\includegraphics[height=3mm]{\pathtographs \hashlookup{#1}}}% 
\end{array}\hspace{-3pt}}}

\newcommand{\Tr}{\operatorname{Tr}}

\newcommand{\jw}[1]{f^{(#1)}}

\usepackage{yfonts}

\title{$1$-supertransitive subfactors with index at most $6 \frac{1}{5}$}
\author{Zhengwei Liu}
\author{Kim Morrison}
\author{David Penneys}

\definecolor{medium-blue}{rgb}{0,0,.8}
\hypersetup{
   colorlinks, linkcolor={purple},
   citecolor={medium-blue}, urlcolor={medium-blue}
}

\newcommand{\fish}{\cB\cH\cF}
\newcommand{\TL}{\cT\hspace{-.08cm}\cL}

\DeclareMathOperator{\tr}{tr}
\DeclareMathOperator{\thick}{thick}
\DeclareMathOperator{\coeff}{coeff}

\usetikzlibrary{calc}
\tikzstyle{unshaded}=[fill=white]
\tikzstyle{shaded}=[fill=red!10!blue!20!gray!30!white]
\newcommand{\nbox}[6]{
	\draw[thick, #1] ($#2+(-#3,-#3)+(-#4,0)$) rectangle ($#2+(#3,#3)+(#5,0)$);
	\coordinate (Zz) at ($#2+(-#4,0)$);
	\coordinate (zZz) at ($#2+(#5,0)$);
	\node at ($1/2*(Zz)+1/2*(zZz)$) {#6};
}
\newcommand{\ncircle}[5]{
	\draw[thick, #1] #2 circle (#3);
	\node at #2 {#5};
	\node at ($#2+(#4:.15cm)+(#4:#3cm)$) {$\star$};
}

\newcommand{\notfree}[1]{}

\begin{document}

\begin{abstract}
We classify irreducible II$_1$ subfactors $A \subset B$ such that $B\ominus A$ is reducible as an $A-A$ bimodule, with
index at most $6\frac{1}{5}$, leaving aside the composite subfactors at index exactly 6.
Previous work has already achieved this up to
index $3+\sqrt{5} \approx 5.23$. We find there are exactly three such subfactors with index in
$(3+\sqrt{5}, 6 \frac{1}{5}]$, all with index $3+2\sqrt{2}$. One of these comes
from $SO(3)_q$ at a root of unity, while the other two appear to be closely related, and are `braided up
to a sign'.
%This is the published version of \arxiv{???}.
\end{abstract}

\maketitle

\section*{Introduction}

A II$_1$-subfactor is an inclusion of infinite dimensional von Neumann algebras $A \subset B$ each with trivial centre, and a trace on $B$ with $\tr(\mathbf 1) = 1$.
The index of the subfactor is the Murray-von Neumann dimension of $B$ as a left $A$-module.
A subfactor is irreducible if $B$ is irreducible as an $A-B$ bimodule. As an $A-A$ bimodule, $B$ is certainly reducible,
with $A$ as a sub-bimodule. In this article, we study the very restricted case in which the index is at most $6
\frac{1}{5}$, $B$ is irreducible as an $A-B$ bimodule, and $B \ominus A$ is reducible as an $A-A$ bimodule (this condition is often called `1-supertransitive').

After Jones' landmark result \cite{MR0696688} that the index of a II$_1$ subfactor is quantized, and must lie in the set 
$$ \set{ 4\cos^2\left(\pi/n\right)}{n \geq 3} \cup [4, \infty],$$ 
Ocneanu and others classified all subfactors with index less than 4 \cite{MR996454}. Popa and others subsequently completeted the classification of subfactors with index exactly 4 \cite{MR1278111}. 
Above index 4, there are subfactors with Temperley-Lieb standard invariant at every index \cite{MR1198815}. 
Leaving these aside, however, the spectrum of possible index values for extremal subfactors remains discrete above 4! 
Haagerup announced that in the index range $(4, 3+\sqrt{3})$ there were at most 3 families of subfactors (besides those with Temperley-Lieb standard invariants) \cite{MR1317352}, and these families were later reduced to exactly 3 subfactors (up to taking duals) which have subsequently been constructed \cite{MR1686551,MR2979509}. 
In a series of recent articles, these classification results were pushed all the way to index less than 5 \cite{MR2914056,MR2902285,MR2993924,MR2902286}. 
Above index 5, we also have a complete classification of 1-supertransitive subfactors with index less than $3+\sqrt{5}$ \cite{1205.2742}. 
For a much more detailed history, including references for all the results mentioned in this paragraph, and an overview of the techniques used in the classification up to index 5, see the survey article \cite{1304.6141}.

Recently, two breakthroughs have open the possibility of pushing classification results to higher indices.
First, \cite{1308.5656} provides results on virtual normalizers implying the existence of an intermediate subfactor.
Second, \cite{MR1334479,1208.1564} shows
that principal graphs which are stable at some depth must remain stable thereafter.

At this point we still have few results on subfactors with arbitrary supertransitivity above index 5. However, this
article explains that the 1-supertransitive case is now completely understood below index 6.
In fact, while index exactly $6$ still holds some mysteries, we can also classify all 1-supertransitive subfactors
with index in the interval $(6, 6\frac{1}{5}]$. We are forced to stop, for now, at $6\frac{1}{5}$, and we explain in
what follows why this index value is an obstacle.

At index exactly $6$, we show that any 1-supertransitive subfactor must have an intermediate subfactor, and is thus a
composition of an $A_3$ subfactor with an $A_5$ subfactor, or of an $A_3$ subfactor with a $D_4$ subfactor. In the
later case, there are (at least) countably many different standard invariants of finite depth subfactors --- indeed every finite
quotient of the modular group $\mathbb Z/ 2\mathbb Z * \mathbb Z / 3\mathbb Z \iso PSL(2,\mathbb Z)$ gives an example of a Bisch-Haagerup
subfactor \cite{MR1386923} at index $6$. But even worse, the result of \cite{1309.5354} (which builds upon \cite{MR2314611}) shows there are
non-isomorphic hyperfinite subfactors, not classifiable by countable structures, all with the same $A_3 * D_4$ standard
invariant at index 6.

At present not much is known about compositions of $A_3$ with $A_5$; unfortunately the techniques
of \cite{1308.5691} and \cite{1308.5723} which classify the standard invariants of compositions of $A_3$ with either $A_3$ or $A_4$
break down.

Our main result is
\begin{thm*}\label{thm:Main}
If $\cP_\bullet$ is a 1-supertransitive subfactor planar algebra in the index range $[3+\sqrt{5},6 \frac{1}{5}]$, then either
$\cP_\bullet$ has a nontrivial biprojection corresponding to an intermediate subfactor (and so the index is exactly
$3+\sqrt{5}$ or exactly 6), or the principal graphs $(\Gamma_+,\Gamma_-)$ are one of
\begin{align*}
\cS&=\left(\bigraph{bwd1v1p1v1x0p1x1p0x1v0x1x0duals1v1x2v1},\bigraph{bwd1v1p1v1x0p1x1p0x1v0x1x0duals1v1x2v1}\right) \text{ or}\\
\cS'&=\left(\bigraph{bwd1v1p1v1x0p1x1p0x1v0x1x0duals1v2x1v1},\bigraph{bwd1v1p1v1x0p1x1p0x1v0x1x0duals1v2x1v1}\right).
\end{align*}
Moreover, there is a unique $\cP_\bullet$ with principal graphs $\cS$, and exactly two $\cP_\bullet$ with principal
graphs $\cS'$ which are complex conjugate.
All three planar algebras are symmetrically self-dual in the sense of \cite[Section 5.1]{1208.3637}, and thus give unshaded planar algebras.
(The red lines on $\Gamma_\pm$ denote dual data. See Section \ref{sec:Notation} for the relevant notation and definitions.)
\end{thm*}

The planar algebra with principal graphs $\cS$ is merely the reduction of $A_7$ at $f^{(2)}$ (equivalently, the shading of the
planar algebra coming from the 3-dimensional representation of quantum $SO(3)$ at the appropriate root of unity). In particular,
it is braided. Although we do not have a direct construction of the other two planar algebras starting from quantum groups,
they appear to be closely related. In particular, while they are not braided, they are ``$\sigma$-braided'', a notion we introduce in Section \ref{sec:Braided}. It would be very interesting to find a construction starting from the braided case, perhaps by some type of
twisting. 

We expect our subfactors to be related to Izumi's $2^41$ subfactors.
By \cite[Examples A.2 and A.3]{MR1832764} (see also \cite[Proposition 6 and Table 2]{1208.1500}), there are exactly one $2^{\bbZ/2\bbZ\times \bbZ/2\bbZ}1$ subfactor and two complex conjugate $2^{\bbZ/4\bbZ}1$ subfactors, with principal graph pairs
$$
\left(
\bigraph{bwd1v1v1p1p1p1v1x0x0x0p0x1x0x0p0x0x1x0duals1v1v1x2x3},
\bigraph{bwd1v1v1p1p1p1v1x0x0x0p0x1x0x0p0x0x1x0duals1v1v1x2x3}
\right)
\text{ and }
\left(
\bigraph{bwd1v1v1p1p1p1v1x0x0x0p0x1x0x0p0x0x1x0duals1v1v2x1x3},
\bigraph{bwd1v1v1p1p1p1v1x0x0x0p0x1x0x0p0x0x1x0duals1v1v2x1x3}
\right)
$$
respectively.
We expect these $2^41$ subfactors arise as equivariantizations of the one $\cS$ and two $\cS'$ subfactors respectively.

This work builds on \cite{MR2914056,1205.2742}, which classify 1-supertransitive subfactors with index in the intervals
$(4,5)$ and $(5, 3+\sqrt{5})$ respectively. Index at most 4 was done by \cite{MR996454,MR1278111} (see \cite{1304.6141} for further details), while index exactly 5 is as yet unpublished, although summarized in \cite{1304.6141}. Subfactor planar algebras with intermediates
at index $3+\sqrt{5}$ were classified in \cite{1308.5691,1308.5723}.

\subsection*{Outline}

In Sections \ref{sec:Weeds} and \ref{sec:main-result}, we show that any 1-supertransitive subfactor with index in $(3+\sqrt{5},6\frac{1}{5}]$ either has principal graphs $\cS$ or $\cS'$, or is an extension of one of the following four graphs:
\begin{align*}
\cD & = \left(\bigraph{bwd1v1p1v1x0p1x1p0x1duals1v1x2},\bigraph{bwd1v1p1v1x0p1x1p0x1duals1v1x2} \right) &
\cK & = \left(\bigraph{bwd1v1p1v1x1p0x1p0x1duals1v1x2},\bigraph{bwd1v1p1v1x1p0x1p0x1duals1v1x2}\right)\\
\cD' & = \left(\bigraph{bwd1v1p1v1x0p1x1p0x1duals1v2x1}, \bigraph{bwd1v1p1v1x0p1x1p0x1duals1v2x1} \right)&
\cK' & = \left(\bigraph{bwd1v1p1v1x1p0x1p0x1duals1v1x2},\bigraph{bwd1v1p1v1x1p1x0p0x1duals1v1x2}\right).
\end{align*}
In Section \ref{sec:Shuriken}, we prove existence and uniqueness (up to duals) of our $\sigma$-braided subfactors with principal graphs $\cS,\cS'$.
Some technical computations from this section are deferred to Appendix \ref{sec:Appendix}.
Finally, in Sections \ref{sec:Dart} and \ref{sec:Kangaroo}, we eliminate extensions of the graphs $\cD,\cD'$ and $\cK,\cK'$ respectively.

\subsection*{Acknowledgements}

This project started during the 2013 NCGOA at Vanderbilt University.
Zhengwei Liu and David Penneys would like to thank Dietmar Bisch and the other organizers.

Kim Morrison was supported by a `Discovery Early Career Research Award' DE120100232 from the Australian Research Council.
David Penneys was partially supported by the Natural Sciences and Engineering Research Council of Canada.
All authors were supported by DOD-DARPA grant HR0011-12-1-0009.

%%%%%%%%%%%%%%%%%%%%%%%%%%%%%%%%%%%%%%%%%%%%%%%%%%%%%%%%%%%%%%%%%%%%%
%%%%%%%%%%%%%%%%%%%%%%%%%%%%%%%%%%%%%%%%%%%%%%%%%%%%%%%%%%%%%%%%%%%%%
%%%%%%%%%%%%%%%%%%%%%%%%%%%%%%%%%%%%%%%%%%%%%%%%%%%%%%%%%%%%%%%%%%%%%
\section{Weeds and intermediate subfactors}\label{sec:Weeds}

%%%%%%%%%%%%%%%%%%%%%%%%%%%%%%%%%%%%%%%%%%%%%%%%%%
\subsection{Notation for planar algebras}\label{sec:Notation}

We refer the reader to \cite{MR2972458,MR2979509} for the definition of a subfactor planar algebra. 
In this article, $\cP_\bullet$ denotes a subfactor planar algebra of modulus $\delta=[2]$ where 
$$
[k]=\frac{q^k-q^{-k}}{q-q^{-1}},
$$ 
and $q\in \set{\exp\left(\frac{2\pi i}{2j}\right)}{j\geq 3}\cup[1,\infty)$ such that $[2]=q+q^{-1}$.
We denote the Temperley-Lieb subfactor planar subalgebra by $\TL_\bullet$.

We refer the reader to \cite{MR999799,1208.1564,1304.6141} for the definition of the principal graphs $(\Gamma_+,\Gamma_-)$ of $\cP_\bullet$.
If there is only one projection $P\in\cP_{n,\pm}$ in the equivalence class $[P]$ corresponding to a vertex of $\Gamma_\pm$ at depth $n$, then we identify $[P]$ with $P$.

For a projection $P\in\cP_{k,\pm}$, the dual projection $\overline{P}$ is given by
$$
\overline{P}
=
\begin{tikzpicture}[baseline=-.1cm]
	\clip (-1.1,-.9)--(-1.1,.9)--(1.1,.9)--(1.1,-.9);	
	\draw (0,.4) arc (180:0:.4cm)--(.8,-.8);
	\draw (0,-.4) arc (0:-180:.4cm)--(-.8,.8);
	\draw[thick, unshaded] (-.4, -.4) -- (-.4, .4) -- (.4, .4) -- (.4,-.4) -- (-.4, -.4);
	\node at (0,0) {$P$};
	\node at (-1,.6) {{\scriptsize{$k$}}};
	\node at (1,-.6) {{\scriptsize{$k$}}};
\end{tikzpicture}.
$$
For a vertex $[P]$ of $\Gamma_\pm$ at depth $n$, there is a corresponding dual vertex $[\overline{P}]$ necessarily at depth $n$.
If $n$ is even then $[\overline{P}]$ is a vertex of $\Gamma_+$, but if $n$ is odd, then $[\overline{P}]$ is a vertex of $\Gamma_-$.

When we draw principal graph pairs, we use the convention that the vertical ordering of vertices at a given odd depth determines the duality; the lowest vertices in each graph at each odd depth are dual to each other, etc.
When we specify the duality at even depths (sometimes we omit this data), the duality is represented by red arcs joining dual pairs of vertices. 
Self-dual even vertices have a small red dash above them.

\begin{defn}
A subfactor planar algebra $\cP_\bullet$ is called \underline{$k$-supertransitive} if $\cP_{j,+}=\TL_{j,+}$ for all $0\leq j\leq k$. Equivalently, $\cP_\bullet$ is $k$-supertransitive if the truncation $\Gamma_+(k)$ of $\Gamma_+$ to depth $k$ is $A_{k+1}$.
This gives us the notion of the supertransitivity of a potential principal graph.
We say a subfactor $N\subset M$ is $k$-supertransitive if its associated subfactor planar algebra is $k$-supertransitive.
\end{defn}

\begin{remark}
When we say a subfactor, subfactor planar algebra, or potential principal graph $\Gamma_\pm$ is $k$-supertransitive, we usually also mean that $\TL_{k+1,+}\subsetneq \cP_{k+1,+}$ or $\Gamma_\pm(k+1)\neq A_{k+2}$, although strictly speaking, this is an abuse of nomenclature.
\end{remark}

%%%%%%%%%%%%%%%%%%%%%%%%%%%%%%%%%%%%%%%%%%%%%%%%%%%%%%%%%%%%%%%%%%%%%
\subsection{Results on 1-supertransitive subfactors using Liu's thickness}\label{sec:Thickness}

Suppose $\cP_\bullet$ is a 1-supertransitive subfactor planar algebra.

\begin{defn}
Suppose $X\in \cP_{2,+}$ is a positive element such that
$$
X=\sum_{i=1}^k c_i P_i
$$
where the $P_i$'s are mutually orthogonal minimal projections in $\cP_{2,+}$, and $c_i>0$ for all $1\leq i\leq k$.
We define the \underline{thickness of $X$}, denoted by $\thick(X)$, to be the number $k$, which is independent of the decomposition of $X$.

For two minimal projections $P,Q\in \cP_{2,+}$, we define the \underline{thickness between $P$ and $Q$}, denoted by $\thick(P,Q)$ to be the number of length 2 paths on $\Gamma_+$ between $[P]$ and $[Q]$.
\end{defn}

The thickness between $P$ and $Q$ is related to the thickness of the positive operator $\overline{P}\circ Q$, where $\circ$ denotes the coproduct on $\cP_{2,+}$:
$$
\overline{P} \circ Q =
\begin{tikzpicture}[baseline = -.1cm]
	\clip (-.5,.8) -- (-.5,-.8) -- (1.5,-.8) -- (1.5,.8);
	\draw[shaded] (-.2,-1) -- (-.2,1) -- (1.2,1) -- (1.2,-1);
	\draw[unshaded] (.2,.4) arc (180:0:.3cm) -- (.8,-.4) arc (0:-180:.3cm);
	\nbox{unshaded}{(0,0)}{.4}{0}{0}{$\overline{P}$}
	\nbox{unshaded}{(1,0)}{.4}{0}{0}{$Q$}
\end{tikzpicture}.
$$

\begin{lem}[{\cite[Lemma 4.2]{1308.5656}}]\label{lem:ThicknessBound}
Suppose $P,Q$ are minimal projections in $\cP_{2,+}$. Then $\thick(\overline{P}\circ Q)\leq \thick(P,Q)$.
\end{lem}

The following is a weak version of a powerful result of Liu on virtual normalizers which ensures the existence of an intermediate subfactor.
This weaker version is sufficient for our purposes.

\begin{thm}[{\cite[weakened Theorem 4.16]{1308.5656}}]
Suppose there is a vertex $[P]$ at depth 2 of $\Gamma_+$ satisfying
\begin{itemize}
\item
$[P]$ is central, i.e., there is only 1 edge to the vertex at depth 1,
\item
$\dim([P])>1$, and
\item
$T(P,Q)=1$ for all $Q\in \cP_{2,+}$ with $[Q]\neq [P]$.
\end{itemize}
Then either $\cP_\bullet$ is Temperley-Lieb or $\cP_\bullet$ has a nontrivial biprojection corresponding to an
intermediate subfactor.
\end{thm}

%%%%%%%%%%%%%%%%%%%%%%%%%%%%%%%%%%%%%%%%%%%%%%%%%%%%%%%%%%%%%%%%%%%%%
\subsection{Results on intermediate subfactors}

If $N\subset M$ is a subfactor with index in $(4,6 \frac{1}{5}]$ with an intermediate subfactor, then $[M: N]\in \{3+\sqrt{5},6\}$.
Subfactor planar algebras with intermediates at index $3+\sqrt{5}$ were completely classified by Liu \cite{1308.5691}
(see also Izumi-Morrison-Penneys \cite{1308.5723}, which simultaneously conjectured the same classification, without
proving all cases):

\begin{thm}[\cite{1308.5691}]
If $N\subset P \subset M$ is an inclusion of II$_1$-factors such that $[P: N]=2$ and $[M:P]=\frac{3+\sqrt{5}}{2}$
(so $[M:N]=3+\sqrt{5}$), then the principal graphs of $N\subset M$ must be one of the following:
\begin{align*}
\fish_1&=
\left(
\bigraph{bwd1v1p1p1v1x1x0duals1v1x2x3},
\bigraph{bwd1v1p1p1v1x1x0duals1v1x2x3}
\right)
\\
\fish_2&=
\left(
\bigraph{bwd1v1p1v1x0p1x0v1x0p1x0v1x1duals1v1x2v1x2},
\bigraph{bwd1v1p1v1x0p0x1v1x0p1x0p1x0p0x1v1x0x0x1v1duals1v1x2v1x2x4x3v1}
\right)
\\
\fish_3&=
\left(
\bigraph{bwd1v1p1v1x0p1x0v1x0v1p1v1x0p1x0v1x1duals1v1x2v1v1x2},
\bigraph{bwd1v1p1v1x0p0x1v1x0p1x0p0x1v1x0x0p0x0x1v1x0p1x0p1x0p0x1p0x1v0x1x0x1x0v1duals1v1x2v1x3x2v1x2x4x3x5v1}
\right)
\\
\cF\cC&=
\left(
\bigraph{bwd1v1p1v1x0p1x0v1x0v1p1v1x0v1p1duals1v1x2v1v1}\cdots,
\bigraph{bwd1v1p1v1x0p0x1v1x0p1x0p0x1v1x0x0p0x0x1v1x0p1x0p0x1p0x1v1x0x0x0p0x0x1x0duals1v1x2v1x3x2v1x3x2x4}\cdots
\right)
\end{align*}
Moreover, there is exactly one subfactor planar algebra with these principal graph pairs.
\end{thm}

\begin{prop}\label{prop:BH}
Suppose $N\subset M$ is a hyperfinite {\rm II}$_1$-subfactor with $(\Gamma_+,\Gamma_-)$ an extension of
$$
\left(\bigraph{bwd1v1p1p1v0x1x0p0x0x1duals1v1x3x2},\bigraph{bwd1v1p1p1v0x0x1p0x0x1duals1v2x1x3}\right).
$$
Then $N\subset M$ is isomorphic to a Bisch-Haagerup subfactor $R^{\Integer/2}\subset R\rtimes \Integer/3$.
\end{prop}
Recall that an extension of $\Gamma$ is a graph of greater depth which truncates to $\Gamma$.
\begin{proof}
Denote the two univalent vertices at depth 2 of $\Gamma_-$ by $\alpha, \beta$. One can show similarly to \cite[Lemma 4.7]{1308.5723} that $e_1+\alpha+\beta$ is a biprojection corresponding to an intermediate subfactor $M\subset Q \subset \langle M, e_N\rangle$ with $[Q:M]=3$.
Denote the self-dual vertex at depth 3 of $\Gamma_-$ by $P$, and note that
$$
P = \id_{2,-}-e_1 -\alpha-\beta.
$$
By Lemma \ref{lem:ThicknessBound}, $P\circ P$ is a positive linear combination of at most 3 minimal projections in $\cP_{2,-}$. A straightforward computation shows
$$
P\circ P = \left(\delta - \frac{6}{\delta}\right) \id_{2,-} + 3(e_1+\alpha+\beta),
$$
so we must have that $\delta-\frac{6}{\delta}=0$, i.e., $\delta^2=6$.

Recall that $M'\cap \langle Q, e_M^Q\rangle$ can be identified with $(e_1+\alpha+\beta)\cP_{2,-}(e_1+\alpha+\beta)$ by \cite{BhattacharyyaLandau}, and $\alpha,\beta$ are fixed by $(e_1+\alpha+\beta)$. Hence $\dim(M'\cap \langle Q, e_M^Q\rangle)=3$, so $M\subset Q$ is isomorphic to the $\Integer/3$-subfactor $R\subset R\rtimes \Integer/3$.

Now $Q$ corresponds to an intermediate subfactor $N\subset P\subset M$ where $[P: N]=2$ and $[M:P]=3$.
By Goldman's Theorem \cite{MR0107827}, $N\subset P$ is isomorphic to the $\Integer/2$-fixed point subfactor $R^{\Integer/2}\subset R$.
Hence $N\subset M$ is a Bisch-Haagerup subfactor of the form $R^{\Integer/2}\subset R\rtimes \Integer/3$.
\end{proof}

\begin{remark}
Subfactors of index 6 with an intermediate subfactor are wild.
By \cite{MR2314611}, there is a continuous family of non-isomorphic, hyperfinite Bisch-Haagerup subfactors at index 6 with the same standard invariant.
In fact, by \cite{1309.5354}, hyperfinite Bisch-Haagerup subfactors at index 6 with standard invariant $A_3*D_4$ are not classifiable by countable structures.
\end{remark}

%%%%%%%%%%%%%%%%%%%%%%%%%%%%%%%%%%%%%%%%%%%%%%%%%%%%%%%%%%%%%%%%%%%%%
\subsection{Results on subfactor planar algebras generated by a single 2-box}

\begin{thm}[Bisch-Jones \cite{MR1972635}]\label{thm:Diamond}
If $\cP_\bullet$ is a subfactor planar algebra generated by a single $2$-box such that $\dim(\cP_{3})=13$, then $\cP_\bullet$ is the subgroup subfactor for $R\rtimes \Integer/2\subset R\rtimes D_5$. Thus $\cP_\bullet$ has principal graphs
$$
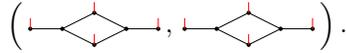
\left(\bigraph{bwd1v1p1v1x1v1duals1v1x2v1},\bigraph{bwd1v1p1v1x1v1duals1v1x2v1}\right).
$$
\end{thm}

\begin{cor}\label{cor:Diamond}
If the principal graph of a subfactor planar algebra $\cP_\bullet$ is an extension of
$$
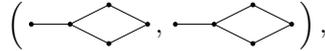
\left(\bigraph{gbg1v1p1v1x1},\bigraph{gbg1v1p1v1x1}\right),
$$
then $\cP_\bullet$ is the planar algebra of the subgroup subfactor $R\rtimes \Integer/2\subset R\rtimes D_5$ with
$$
(\Gamma_+,\Gamma_-)=\left(\bigraph{bwd1v1p1v1x1v1duals1v1x2v1},\bigraph{bwd1v1p1v1x1v1duals1v1x2v1}\right).
$$
\end{cor}
\begin{proof}
Let $\cQ_\bullet$ be the planar subalgebra of $\cP_\bullet$ generated by $\cP_{2,+}$.
Then $\dim(\cQ_{2,+})\leq 13$, so $\cQ_{\bullet}$ is one of the following subfactors by \cite{MR1733737,MR1972635}:
\begin{enumerate}
\item the group subfactor $R\subset R\rtimes \Integer/3$,
\item the subgroup subfactor $R\rtimes \Integer/2\subset R\rtimes D_5$, or
\item Fuss-Catalan.
\end{enumerate}
The first is impossible by the Frobenius-Perron dimensions, and if $\cQ_\bullet$ is the second, we are finished.
Note that the dimensions of the bimodules at depth 2 must be equal by the shape of the graph, so the branch factor must be $r=1$.
But this is not the case for any Fuss-Catalan planar algebra by \cite{1307.5890}, since Fuss-Catalan has annular multiplicities $*10$ or $*11$.
\end{proof}

\begin{thm}[{\cite{LiuSinglyGenerated}}]\label{thm:BMW}
If $\cP_\bullet$ is a subfactor planar algebra generated by a single $2$-box such that $\dim(\cP_{3})=14$, then the principal graphs of $\cP_\bullet$ are extensions of
$$
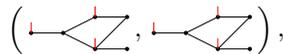
\left( \bigraph{bwd1v1p1v1x0p1x1duals1v1x2}, \bigraph{bwd1v1p1v1x0p1x1duals1v1x2}\right),
$$
and $\cP_\bullet$ is a BMW planar algebra in the sense of  \cite{MR1090432}.
In particular, $\cP_\bullet$ is either the depth 3 trivial extension coming from $O(3,\Real)$, or $\cP_\bullet$ comes from $Sp(4,\Real)$ (the case $n=-5$ where $r=q^n$ and $q=e^{\pi i/\ell}$ with $\ell\geq 12$ even \cite[Table 1]{MR1090432}), and has index
$$
\left(\frac{\sin\left(|n| \frac{\pi}{\ell}\right)}{\sin\left(\frac{\pi}{\ell}\right)}-1\right)^2 > 7.4641.
$$
\end{thm}

\begin{remark}
Note that the $q$ for the BMW planar algebras is \underline{not} the same $q$ such that $[2]=q+q^{-1}$.
\end{remark}

\begin{cor}\label{cor:BMW}
If the principal graph of a subfactor planar algebra $\cP_\bullet$ is an extension of
$$
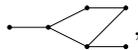
\bigraph{gbg1v1p1v1x0p1x1},
$$
and $\delta^2<7.4641$, then $\cP_\bullet$ is the unique self-dual subfactor planar algebra with
$$
(\Gamma_+,\Gamma_-)=
\left(
\bigraph{bwd1v1p1v1x0p1x1duals1v1x2},
\bigraph{bwd1v1p1v1x0p1x1duals1v1x2}
\right)
$$
(uniqueness and self-duality follows from \cite{MR1090432,1205.2742,LiuSinglyGenerated}).
\end{cor}
\begin{proof}
Let $\cQ_\bullet$ be the planar subalgebra of $\cP_\bullet$ generated by $\cP_{2,+}$.
Then $\dim(\cQ_{3,+})\leq 14$.
If $\dim(\cQ_{3,+})=14$, then $\cQ_{3,+}=\cP_{3,+}$, and by Theorem \ref{thm:BMW} and the index restriction, $\cQ_\bullet$ must be the unique subfactor planar algebra with principal graphs
$$
\left(
\bigraph{bwd1v1p1v1x0p1x1duals1v1x2},
\bigraph{bwd1v1p1v1x0p1x1duals1v1x2}
\right).
$$
In this case, since $\cQ_\bullet$ is finite depth, we know $\cP_\bullet$ has depth at most the depth of $\cQ_\bullet$ (e.g., by \cite[Corollary 3.11]{1208.1564}), so $\cP_\bullet=\cQ_\bullet$.

Otherwise, $\dim(\cQ_{3,+})\leq 13$, and $\cQ_{\bullet}$ is one of the 3 subfactors in the proof of Corollary \ref{cor:Diamond}. Cases (1) and (2) are impossible by the Frobenius-Perron dimensions.
If $\cQ_\bullet$ is Fuss-Catalan $A_m * A_n$, then by examining the Frobenius-Perron dimensions at depth 2, $m,n>3$.
By the index restriction, we must have $m=n=4$, so $\delta^2=(4\cos^2(\frac{\pi}{5}))^2\approx 6.8541$.
Hence $\cP_\bullet$ has a biprojection corresponding to an intermediate subfactor, so it must be a composed inclusion of two $A_4$ subfactors. 
But these are completely classified by \cite{1308.5691}, so this case is impossible.
\end{proof}

%%%%%%%%%%%%%%%%%%%%%%%%%%%%%%%%%%%%%%%%%%%%%%%%%%%%%%%%%%%%%%%%%%%%%
\subsection{Results on principal graph stability}

%Let $\cP_\bullet$ be a subfactor planar algebra of modulus $\delta$ with principal graphs $(\Gamma_+,\Gamma_-)$, and let $\Gamma_\pm(k)$ denote the truncation of $\Gamma_\pm$ to depth $k$.
Recall the following definition from \cite{MR1334479,1208.1564}.

\begin{defn}
A principal graph $\Gamma$ is said to be \emph{stable at depth $n$} if every vertex at depth $n$ connects to at most one vertex at depth $n+1$, no two vertices at depth $n$ connect to the same vertex at depth $n+1$, and all edges between depths $n$ and $n+1$ are simple.
We say $(\Gamma_+,\Gamma_-)$  is \emph{stable at depth $n$} if both $\Gamma_+$ and $\Gamma_-$ are stable at depth $n$.
\end{defn}

\begin{thm}[{Principal Graph Stability \cite[Theorem 4.5]{MR1334479}, \cite[Theorem 1.3]{1208.1564}}]\label{thm:PrincipalGraphStability}
\mbox{}
\begin{enumerate}[(1)]
\item
If $(\Gamma_+,\Gamma_-)$ is stable at depth $n$, the truncation $\Gamma_\pm(n+1)\neq A_{n+2}$, and $\delta>2$, then $(\Gamma_+,\Gamma_-)$ is stable at depth $k$ for all $k\geq n$, and $\Gamma_+,\Gamma_-$ are finite.
\item
If $\Gamma_+$ is stable at depths $n$ and $n+1$, the truncation $\Gamma_+(n+1)\neq A_{n+2}$, and $\delta>2$, then $(\Gamma_+,\Gamma_-)$ is stable at depth $k$ for all $k\geq n+1$, and $\Gamma_+,\Gamma_-$ are finite.
\end{enumerate}
\end{thm}

\begin{defn}
A \underline{cylinder} is a stable principal graph, i.e., it is a principal graph $(\Gamma_+,\Gamma_-)$ such that $(\Gamma_+,\Gamma_-)$ is stable at depth $n$ for some $n$.
\end{defn}

Recall from \cite{MR2914056} that we prove `classification statements' $(\Gamma, \delta^2,  \cW, \cV)$ describing all the possible translated extensions with graph norm at most $\delta$ of a given graph $\Gamma$, by specifying a collection of \emph{weeds} $\cW$ and a collection of \emph{vines} $\cV$. The assertion of a classification statement is that every translated extension of $\Gamma$ with graph norm at most $\delta$ which could be the principal graph of a subfactor is either a translated extension of a weed or a translation of a vine. 

We improve classification statements by taking a weed, moving it into the set of vines, and then adding all the depth 1 extensions of that weed to the set of weeds. The consequence of Theorem \ref{thm:PrincipalGraphStability} is that if a weed is  a cylinder we only need to consider stable extensions.

\begin{defn}
Suppose $\cP_\bullet$ is a subfactor planar algebra with principal graphs $(\Gamma_+,\Gamma_-)$.
A \underline{tail} of $\Gamma_+$ is a chain of vertices with degree at most $2$ connected to some vertex $v$ of $\Gamma$.
\end{defn}

The next proposition is probably folklore.

\begin{prop}\label{prop:ShortTails}
Suppose $\cP_\bullet$ is a finite depth subfactor planar algebra with principal graph $\Gamma_+$. Then no tail of $\Gamma_+$ may be longer than the initial arm.
\end{prop}
\begin{proof}
Suppose $\cP_\bullet$ is $n-1$-supertransitive and $\Gamma_+$ has a tail of length $m$.
Let $P\in\cP_{d,+}$ be a representative of the vertex $[P]$ at the end of the tail, where $[P]$ has depth $d$.
An induction argument shows that for each $k=1,\dots, m$, $P\otimes \jw{k}$ is a representative of the vertex at distance $m-k$ on the tail.
$$
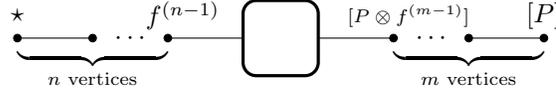
\begin{tikzpicture}
	\draw (0,0) -- (1,0);
	\draw (2,0) -- (3,0);
	\node at (1.5,0) {$\cdots$};
	\filldraw (0,0) circle (.05cm);
	\node at (0,.3) {$\star$};
	\node at (1,-.4) {$\underbrace{\hspace{2cm}}_{n\text{ vertices}}$};
	\filldraw (1,0) circle (.05cm);
	\filldraw (2,0) circle (.05cm);
	\node at (2.2,.3) {$\jw{n-1}$};
	\draw[rounded corners=5pt, very thick]  (3,-.5) rectangle (4,.5);
	\filldraw (5,0) circle (.05cm);
	\node at (5.2,.3) {\scriptsize{$[P\otimes \jw{m-1}]$}};
	\filldraw (6,0) circle (.05cm);
	\filldraw (7,0) circle (.05cm) node [above] {$[P]$};
	\node at (6,-.4) {$\underbrace{\hspace{2cm}}_{m\text{ vertices}}$};
	\draw (4,0) -- (5,0);
	\draw (6,0) -- (7,0);
	\node at (5.5,0) {$\cdots$};
\end{tikzpicture}
$$
Since $\jw{n}$ is not simple, $P\otimes \jw{n}$ is not simple, so $m\leq n$.
\end{proof}

\begin{cor}\label{cor:1STCylinder}
Suppose the cylinder $(\Gamma_+,\Gamma_-)$ has fixed supertransitivity $n-1$.
Then any tail of $\Gamma_\pm$ has length at most $n$.
In particular, a cylinder with fixed supertransitivity 1 cannot be extended.
\end{cor}

In the language of classification statements of \cite{MR2914056}, because we are only interested in 1-supertransitive subfactors, we can move any cylindrical weeds to the vines.

%%%%%%%%%%%%%%%%%%%%%%%%%%%%%%%%%%%%%%%%%%%%%%%%%%%%%%%%%%%%%%%%%%%%%
%%%%%%%%%%%%%%%%%%%%%%%%%%%%%%%%%%%%%%%%%%%%%%%%%%%%%%%%%%%%%%%%%%%%%
%%%%%%%%%%%%%%%%%%%%%%%%%%%%%%%%%%%%%%%%%%%%%%%%%%%%%%%%%%%%%%%%%%%%%
\section{Proof of the main result}
\label{sec:main-result}

\begin{thm}\label{thm:MainHelper}
Suppose $N\subset M$ is a 1-supertransitive subfactor with no intermediate subfactor and index in the range $[3+\sqrt{5},6 \frac{1}{5}]$.
Then the principal graphs $(\Gamma_+,\Gamma_-)$ are extensions of one of the following 4 weeds:
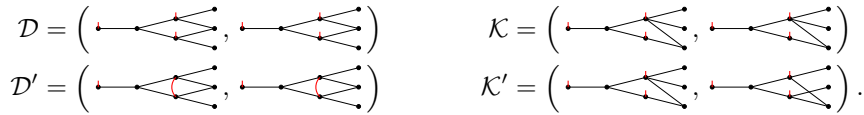
\begin{align*}
\cD & = \left(\bigraph{bwd1v1p1v1x0p1x1p0x1duals1v1x2},\bigraph{bwd1v1p1v1x0p1x1p0x1duals1v1x2} \right) &
\cK & = \left(\bigraph{bwd1v1p1v1x1p0x1p0x1duals1v1x2},\bigraph{bwd1v1p1v1x1p0x1p0x1duals1v1x2}\right)\\
\cD' & = \left(\bigraph{bwd1v1p1v1x0p1x1p0x1duals1v2x1}, \bigraph{bwd1v1p1v1x0p1x1p0x1duals1v2x1} \right)&
\cK' & = \left(\bigraph{bwd1v1p1v1x1p0x1p0x1duals1v1x2},\bigraph{bwd1v1p1v1x1p1x0p0x1duals1v1x2}\right).
\end{align*}
\end{thm}
\begin{proof}

We run the principal graph odometer of \cite{MR2914056} on the weed $(A_2,A_2)=\left( \smallbigraph{bwd1duals1},\smallbigraph{bwd1duals1}\right)$, noting that we may ignore the following weeds:
\begin{itemize}
\item
$(A_3,A_3)=\left( \smallbigraph{bwd1v1duals1v1},\smallbigraph{bwd1v1duals1v1}\right)$,
since we are only concerned with 1-supertransitive subfactors,
\item
$\left(\bigraph{bwd1v1p1v1x1duals1v2x1},\bigraph{bwd1v1p1v1x1duals1v2x1}\right)$, since it has no extensions by Theorem \ref{thm:Diamond},
\item
$\left(\bigraph{bwd1v1p1v1x1duals1v1x2},\bigraph{bwd1v1p1v1x1duals1v1x2}\right)$,
since its extensions are classified by Corollary \ref{cor:Diamond},
\item
$\left(\bigraph{bwd1v1p1v1x0p1x1duals1v2x1},\bigraph{bwd1v1p1v1x0p1x1duals1v2x1}\right)$, since the bottom vertex at depth 3 must have dimension 0, which is impossible,
\item
$\left(\bigraph{bwd1v1p1v1x0p1x1duals1v1x2},\bigraph{bwd1v1p1v1x0p1x1duals1v1x2}\right)$,
since its extensions with index at most 7.4641 are classified by Corollary \ref{cor:BMW},
\item
$\left(\bigraph{bwd1v1p1p1v0x1x0p0x0x1duals1v1x3x2},\bigraph{bwd1v1p1p1v0x0x1p0x0x1duals1v2x1x3}\right)$,
since any extension must be a Bisch-Haagerup subfactor of index 6 with an intermediate subfactor by Proposition \ref{prop:BH},
\item
$\left(\bigraph{bwd1v1p1p1v0x1x0p0x0x1duals1v1x2x3},\bigraph{bwd1v1p1p1v0x0x1p0x0x1duals1v2x1x3}\right)$,
since analyzing the rotation by $\pi$ implies $\Gamma_+$ and $\Gamma_-$ must have the same number of self-dual vertices one past the branch by \cite[Lemma 3.6]{MR2914056}.
\end{itemize}

We get 300 weeds at depth 3, of which 292 have intermediates by the results of Section \ref{sec:Thickness}.
Of the remaining 8, two are stable in the sense of \cite{MR1334479,1208.1564}, leaving only 6 weeds that will continue to grow as we keep running the odometer, along with 19 vines and 2 cylinders.
Four of these weeds are $\cD,\cD',\cK,\cK'$, which we ignore as we continue to run the odometer.
We then run the odometer to depth 5, leaving no remaining weeds which are not $\cD,\cD',\cK,\cK'$, and a total of 32 vines and cylinders.

We present the output of the odometer in Figure \ref{fig:TikzTree5Pruned}, omitting weeds with intermediates.
The {\tt Mathematica} notebook performing this computation can be found as {\tt
1STBelow6Enumerator.nb} with the {\tt arXiv} sources of this article. (It
further relies on the {\tt FusionAtlas`} package; see \cite{MR2914056} for more
details.) This notebook also contains sections for the later parts of the paper
which enumerate principal graphs.

We highlight cylinders with a blue background and `active' weeds with a red background.
Weeds with a white background which are leaves on the tree do not have descendants at the next depth.
See the description in \cite{MR2914056} and \cite{1205.2742} for more details.

\afterpage{
\begin{figure}[!t]
\resizebox{!}{0.95\textheight}{
\begin{tikzpicture}
\tikzset{grow=right,level distance=130pt}
\tikzset{every tree node/.style={draw,fill=white,rectangle,rounded corners,inner sep=2pt}}
\Tree
[.\node{$\!\!\begin{array}{c}\bigraph{bwd1duals1}\\\bigraph{bwd1duals1}\end{array}\!\!$};
	[.\node{$\!\!\begin{array}{c}\smallbigraph{bwd1v2duals1v1}\\\smallbigraph{bwd1v2duals1v1}\end{array}\!\!$};
		[.\node[fill=blue!30]{$\!\!\begin{array}{c}\smallbigraph{bwd1v2v1duals1v1}\\\smallbigraph{bwd1v2v1duals1v1}\end{array}\!\!$};]]
	[.\node{$\!\!\begin{array}{c}\smallbigraph{bwd1v2duals1v1}\\\bigraph{bwd1v1p1p1p1duals1v2x1x4x3}\end{array}\!\!$};]
	[.\node{$\!\!\begin{array}{c}\smallbigraph{bwd1v2duals1v1}\\\bigraph{bwd1v1p1p1p1duals1v2x1x3x4}\end{array}\!\!$};]
	[.\node{$\!\!\begin{array}{c}\smallbigraph{bwd1v2duals1v1}\\\bigraph{bwd1v1p1p1p1duals1v1x2x3x4}\end{array}\!\!$};]
	[.\node{$\!\!\begin{array}{c}\bigraph{bwd1v1p1duals1v2x1}\\\bigraph{bwd1v1p1duals1v2x1}\end{array}\!\!$};
		[.\node[fill=red!30]{$\!\!\begin{array}{c}\bigraph{bwd1v1p1v1x0p1x1p0x1duals1v2x1}\\\bigraph{bwd1v1p1v1x0p1x1p0x1duals1v2x1}\end{array}\!\!$};]]
	[.\node{$\!\!\begin{array}{c}\bigraph{bwd1v2p1duals1v1x2}\\\bigraph{bwd1v2p1duals1v1x2}\end{array}\!\!$};
		[.\node[fill=blue!30]{$\!\!\begin{array}{c}\bigraph{bwd1v2p1v0x1duals1v1x2}\\\bigraph{bwd1v2p1v0x1duals1v1x2}\end{array}\!\!$};]]
	[.\node{$\!\!\begin{array}{c}\bigraph{bwd1v2p1duals1v1x2}\\\bigraph{bwd1v1p1p1p1p1duals1v2x1x4x3x5}\end{array}\!\!$};]
	[.\node{$\!\!\begin{array}{c}\bigraph{bwd1v2p1duals1v1x2}\\\bigraph{bwd1v1p1p1p1p1duals1v2x1x3x4x5}\end{array}\!\!$};]
	[.\node{$\!\!\begin{array}{c}\bigraph{bwd1v2p1duals1v1x2}\\\bigraph{bwd1v1p1p1p1p1duals1v1x2x3x4x5}\end{array}\!\!$};]
	[.\node{$\!\!\begin{array}{c}\bigraph{bwd1v1p1duals1v1x2}\\\bigraph{bwd1v1p1duals1v1x2}\end{array}\!\!$};
		[.\node[fill=red!30]{$\!\!\begin{array}{c}\bigraph{bwd1v1p1v1x1p1x0p0x1duals1v1x2}\\\bigraph{bwd1v1p1v1x1p0x1p1x0duals1v1x2}\end{array}\!\!$};]
		[.\node[fill=red!30]{$\!\!\begin{array}{c}\bigraph{bwd1v1p1v1x0p1x1p1x0duals1v1x2}\\\bigraph{bwd1v1p1v1x0p1x1p0x1duals1v1x2}\end{array}\!\!$};]
		[.\node[fill=red!30]{$\!\!\begin{array}{c}\bigraph{bwd1v1p1v1x1p0x1p0x1duals1v1x2}\\\bigraph{bwd1v1p1v1x1p1x0p1x0duals1v1x2}\end{array}\!\!$};]
		[.\node{$\!\!\begin{array}{c}\bigraph{bwd1v1p1v1x0p1x0p1x1p0x1duals1v1x2}\\\bigraph{bwd1v1p1v1x0p1x0p1x1p0x1duals1v1x2}\end{array}\!\!$};
			[.\node[fill=blue!30]{$\!\!\begin{array}{c}\bigraph{bwd1v1p1v1x0p1x0p1x1p0x1v0x0x0x1duals1v1x2v1}\\\bigraph{bwd1v1p1v1x0p1x0p1x1p0x1v0x0x0x1duals1v1x2v1}\end{array}\!\!$};]
			[.\node[fill=blue!30]{$\!\!\begin{array}{c}\bigraph{bwd1v1p1v1x0p1x0p1x1p0x1v1x0x0x0duals1v1x2v1}\\\bigraph{bwd1v1p1v1x0p1x0p1x1p0x1v1x0x0x0duals1v1x2v1}\end{array}\!\!$};]
			[.\node[fill=blue!30]{$\!\!\begin{array}{c}\bigraph{bwd1v1p1v1x0p1x0p1x1p0x1v1x0x0x0p0x1x0x0duals1v1x2v2x1}\\\bigraph{bwd1v1p1v1x0p1x0p1x1p0x1v1x0x0x0p0x1x0x0duals1v1x2v2x1}\end{array}\!\!$};]
			[.\node{$\!\!\begin{array}{c}\bigraph{bwd1v1p1v1x0p1x0p1x1p0x1v0x0x0x1p0x0x0x1duals1v1x2v2x1}\\\bigraph{bwd1v1p1v1x0p1x0p1x1p0x1v0x0x0x1p0x0x0x1duals1v1x2v2x1}\end{array}\!\!$};]
			[.\node{$\!\!\begin{array}{c}\bigraph{bwd1v1p1v1x0p1x0p1x1p0x1v0x0x0x1p0x0x0x1duals1v1x2v1x2}\\\bigraph{bwd1v1p1v1x0p1x0p1x1p0x1v0x0x0x1p0x0x0x1duals1v1x2v2x1}\end{array}\!\!$};]
			[.\node[fill=blue!30]{$\!\!\begin{array}{c}\bigraph{bwd1v1p1v1x0p1x0p1x1p0x1v1x0x0x0p0x1x0x0duals1v1x2v1x2}\\\bigraph{bwd1v1p1v1x0p1x0p1x1p0x1v1x0x0x0p0x1x0x0duals1v1x2v1x2}\end{array}\!\!$};]
			[.\node[fill=blue!30]{$\!\!\begin{array}{c}\bigraph{bwd1v1p1v1x0p1x0p1x1p0x1v0x1x0x0p0x0x0x1duals1v1x2v1x2}\\\bigraph{bwd1v1p1v1x0p1x0p1x1p0x1v0x1x0x0p0x0x0x1duals1v1x2v1x2}\end{array}\!\!$};]
			[.\node{$\!\!\begin{array}{c}\bigraph{bwd1v1p1v1x0p1x0p1x1p0x1v0x0x0x1p0x0x0x1duals1v1x2v1x2}\\\bigraph{bwd1v1p1v1x0p1x0p1x1p0x1v0x0x0x1p0x0x0x1duals1v1x2v1x2}\end{array}\!\!$};]]
		[.\node{$\!\!\begin{array}{c}\bigraph{bwd1v1p1v1x0p1x1p1x0p0x1duals1v1x2}\\\bigraph{bwd1v1p1v1x0p1x1p0x1p0x1duals1v1x2}\end{array}\!\!$};
			[.\node[fill=blue!30]{$\!\!\begin{array}{c}\bigraph{bwd1v1p1v1x0p1x1p0x1p0x1v1x0x0x0p0x0x1x0duals1v1x2v2x1}\\\bigraph{bwd1v1p1v1x0p1x1p1x0p0x1v0x0x1x0p0x0x0x1duals1v1x2v2x1}\end{array}\!\!$};]]]
	[.\node{$\!\!\begin{array}{c}\bigraph{bwd1v1p1p1duals1v2x1x3}\\\bigraph{bwd1v1p1p1duals1v2x1x3}\end{array}\!\!$};]
	[.\node{$\!\!\begin{array}{c}\bigraph{bwd1v1p1p1duals1v1x2x3}\\\bigraph{bwd1v1p1p1duals1v1x2x3}\end{array}\!\!$};]
	[.\node{$\!\!\begin{array}{c}\bigraph{bwd1v1p1p1p1duals1v2x1x4x3}\\\bigraph{bwd1v1p1p1p1duals1v2x1x4x3}\end{array}\!\!$};]
	[.\node{$\!\!\begin{array}{c}\bigraph{bwd1v1p1p1p1duals1v2x1x3x4}\\\bigraph{bwd1v1p1p1p1duals1v2x1x3x4}\end{array}\!\!$};]
	[.\node{$\!\!\begin{array}{c}\bigraph{bwd1v1p1p1p1duals1v1x2x3x4}\\\bigraph{bwd1v1p1p1p1duals1v1x2x3x4}\end{array}\!\!$};]
	[.\node{$\!\!\begin{array}{c}\bigraph{bwd1v1p1p1p1p1duals1v2x1x4x3x5}\\\bigraph{bwd1v1p1p1p1p1duals1v2x1x4x3x5}\end{array}\!\!$};]
	[.\node{$\!\!\begin{array}{c}\bigraph{bwd1v1p1p1p1p1duals1v2x1x3x4x5}\\\bigraph{bwd1v1p1p1p1p1duals1v2x1x3x4x5}\end{array}\!\!$};]
	[.\node{$\!\!\begin{array}{c}\bigraph{bwd1v1p1p1p1p1duals1v1x2x3x4x5}\\\bigraph{bwd1v1p1p1p1p1duals1v1x2x3x4x5}\end{array}\!\!$};]]
\end{tikzpicture}
}
\caption{Results of running the odometer on $(A_2,A_2)$ to depth 5, ignoring $\cD,\cD',\cK,\cK'$.}
\label{fig:TikzTree5Pruned}
\end{figure}
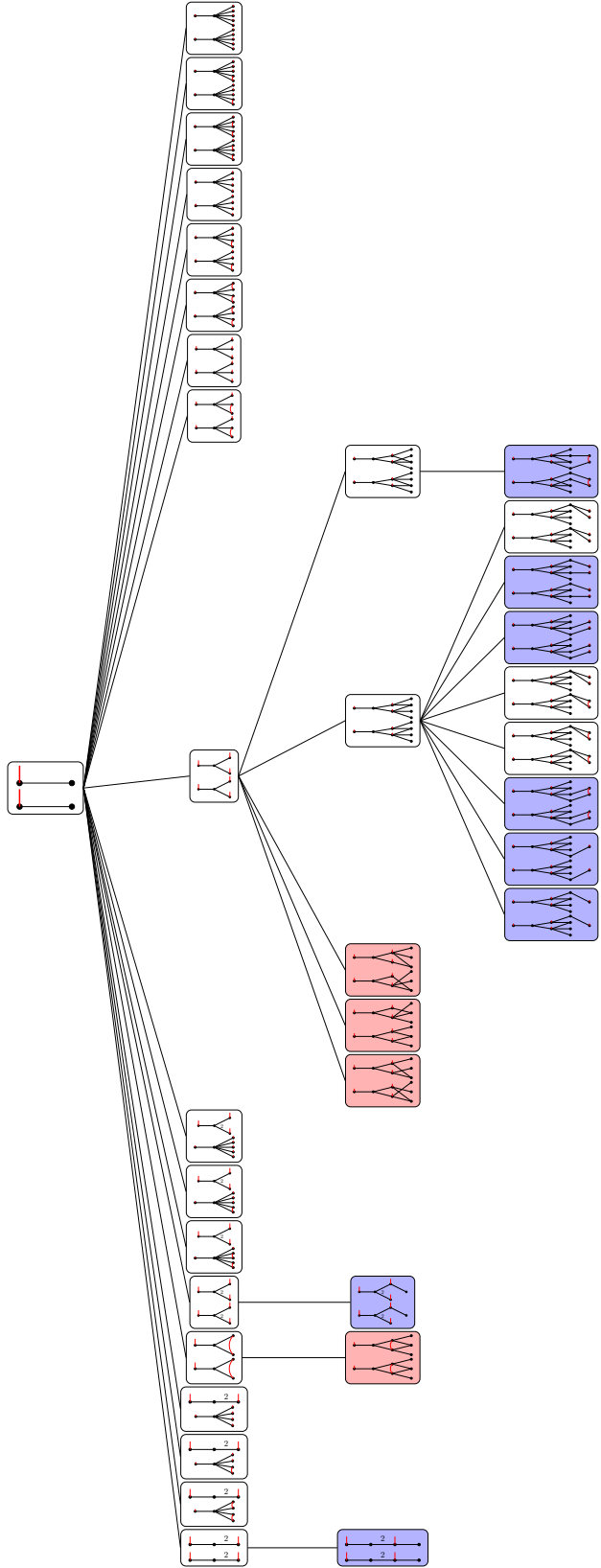
\clearpage
}

Since we are only interested in 1-supertransitive subfactors, the vines cannot be translated.
By Proposition \ref{prop:ShortTails} and Corollary \ref{cor:1STCylinder}, the tails of the cylinders can have length at most 1, and hence cannot be extended. Thus from this point we consider all the cylinders as vines.

We now examine the vines in the index range $[3+\sqrt{5},6\frac{1}{2}]$.
All the vines either have a vertex whose dimension is less than 1 or have a vertex whose dimension is not an algebraic integer (ruling them out), or they are on the following list of potential principal graphs.
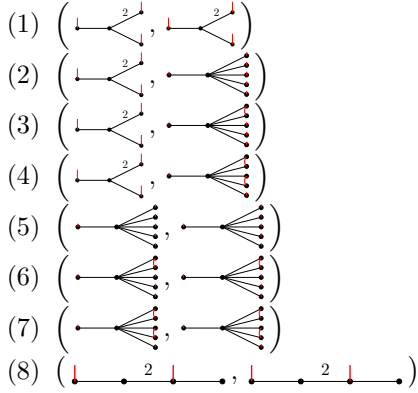
\begin{enumerate}
\item $\left(\bigraph{bwd1v1p2duals1v1x2},\bigraph{bwd1v1p2duals1v1x2}\right)$
\item $\left(\bigraph{bwd1v1p2duals1v1x2},\bigraph{bwd1v1p1p1p1p1duals1v1x2x3x4x5}\right)$
\item $\left(\bigraph{bwd1v1p2duals1v1x2},\bigraph{bwd1v1p1p1p1p1duals1v1x2x3x5x4}\right)$
\item $\left(\bigraph{bwd1v1p2duals1v1x2},\bigraph{bwd1v1p1p1p1p1duals1v1x3x2x5x4}\right)$
\item $\left(\bigraph{bwd1v1p1p1p1p1duals1v1x2x3x4x5},\bigraph{bwd1v1p1p1p1p1duals1v1x2x3x4x5}\right)$
\item $\left(\bigraph{bwd1v1p1p1p1p1duals1v1x2x3x5x4},\bigraph{bwd1v1p1p1p1p1duals1v1x2x3x5x4}\right)$
\item $\left(\bigraph{bwd1v1p1p1p1p1duals1v1x3x2x5x4},\bigraph{bwd1v1p1p1p1p1duals1v1x3x2x5x4}\right)$
\item $\left(\smallbigraph{bwd1v2v1duals1v1},\smallbigraph{bwd1v2v1duals1v1}\right)$
\end{enumerate}
Note that (3) is the principal graph pair of the $S_3$-subfactor, and (7) is the principal graph pair of the $\Integer/6$-subfactor.
In fact, depth 2 subfactors are classified by outer actions of finite dimensional Kac algebras \cite{MR996454,MR1186139}.
We know all such algebras of dimension 6; they are either group algebras or duals of group algebras by \cite{MR1662308}.
Thus, (1), (2), (4), (5), and (6) are not principal graph pairs of subfactors.
By \cite[p. 991]{MR1145672}, (8) is not the principal graph pair of a subfactor.
\end{proof}

Using the above theorem, we now prove our main result.

\begin{proof}[Proof of the main Theorem \ref{thm:Main}]
In Section \ref{sec:Shuriken}, we show that there is exactly one subfactor planar algebra with principal graphs $\cS$, and exactly two subfactors with principal graphs $\cS'$ which are complex conjugate, and all three are symmetrically self-dual in the sense of \cite[Section 5.1]{1208.3637}.
The weeds $\cD,\cD'$ are eliminated in Section \ref{sec:Dart}.
The weeds $\cK,\cK'$ are eliminated in Section \ref{sec:Kangaroo}.
\end{proof}

%%%%%%%%%%%%%%%%%%%%%%%%%%%%%%%%%%%%%%%%%%%%%%%%%%%%%%%%%%%%%%%%%%%%%
%%%%%%%%%%%%%%%%%%%%%%%%%%%%%%%%%%%%%%%%%%%%%%%%%%%%%%%%%%%%%%%%%%%%%%%%%%%%%%%%%%%%%%%%%%%%%%%%%%%%%%%%%%%%%%%%%%%%%%%%%%%%%%%%%%%%%%%%%%
\section{Classification of subfactors with principal graphs
\texorpdfstring{$\cS,\cS'$}{S, S'}}\label{sec:Shuriken}

In this section, we construct one subfactor with principal graphs $\cS$ and two subfactors with principal graphs $\cS'$ which are complex conjugate, where
\begin{align*}
\cS&=\left(\bigraph{bwd1v1p1v1x0p1x1p0x1v0x1x0duals1v1x2v1},\bigraph{bwd1v1p1v1x0p1x1p0x1v0x1x0duals1v1x2v1}\right) \\
\cS'&=\left(\bigraph{bwd1v1p1v1x0p1x1p0x1v0x1x0duals1v2x1v1},\bigraph{bwd1v1p1v1x0p1x1p0x1v0x1x0duals1v2x1v1}\right).
\end{align*}
We then show these are the unique such subfactors with these principal graphs.

To do so, we work in the common underlying bipartite graph planar algebra of $\cS,\cS'$ (without dual data).
By the symmetry of the graph, the vertices at depth 2 have the same dimension.
Thus, up to a scalar, the new low-weight rotational eigenvector at depth 2 corresponding to eigenvalue $\omega=\pm 1$ must be the difference of the two minimal projections $P,Q$ at depth 2.
Note that
\begin{itemize}
\item
$P,Q$ are self-dual if and only if $\omega=1$, and
\item
$P,Q$ are dual to each other if and only if $\omega=-1$.
\end{itemize}

For $\omega=1$, we find a 1-parameter family of self-adjoint, uncappable elements which square to $\jw{2}$, while for $\omega=-1$ we find two such families.
We then show that related elements satisfy a relation analogous to the braid relation, and that these relations are independent of the parameter. (Thus each 1-parameter family actually gives a single planar algebra, and the parameter only determines the embedding of this planar algebra in the graph planar algebra.) Using this relation, we see that planar algebras generated by these elements are evaluable, and hence subfactor planar algebras.

%%%%%%%%%%%%%%%%%%%%%%%%%%%%%%%%%%%%%%%%%%%%%%%%%%%%%%%%%%%%%%%%%%%%%
\subsection{\texorpdfstring{$\sigma$}{sigma}-braided planar algebras}\label{sec:Braided}

By Theorem \ref{thm:Diamond} and Theorem \ref{thm:BMW}, subfactor planar algebras $\cP_\bullet$ generated by a single $2$-box such that $\dim(\cP_{3,\pm})\leq 14$ are completely classified.
One could hope to extend this classification to dimension 15, where it is possible to simplify at least one triangle formed by the 2-box \cite[p. 33]{math.QA/9909027}.

Note that if $\cP_\bullet$ is a subfactor planar algebra with principal graphs $\cS$ or $\cS'$, then there is a new low-weight rotational eigenvector at depth 2, and $\dim(\cP_{3,\pm})= 15$.
In the examples at hand, we can use this element to build a quadratic $\sigma$-braiding, defined below.

\begin{defn}
Let
$
\cF=
\begin{tikzpicture}[baseline=-.1cm]
	\draw (0,-.8)--(0,.8);
	\node at (0,1) {\scriptsize{$n-1$}};
	\node at (0,-1) {\scriptsize{$n-1$}};
	\draw (.2,.4) arc (180:0:.2cm) -- (.6,-1);
	\draw (-.2,-.4) arc (0:-180:.2cm) -- (-.6,1);
	\nbox{unshaded}{(0,0)}{.4}{0}{0}{}
\end{tikzpicture}: \cP_{n,\pm}\to \cP_{n,\mp}
$
be the one-click rotation.
\end{defn}

\begin{defn}
A 2-box $R$ is called a \emph{quadratic $\sigma$-braiding} if
\begin{enumerate}[(1)]
\item
$R$ is bi-invertible, i.e.,
$
\begin{tikzpicture}[baseline=.5cm]
	\coordinate (a) at (0,0);
	\coordinate (b) at (0,1);
	\draw ($(a)+(-.2,-.7)$) -- ($(b)+(-.2,.7)$);
	\draw ($(a)+(.2,-.7)$) -- ($(b)+(.2,.7)$);
	\ncircle{unshaded}{(a)}{.4}{180}{$R$}
	\ncircle{unshaded}{(b)}{.4}{180}{$R^{-1}$}
\end{tikzpicture}
=
\id_{\cP_{2,+}}
$
and
$
\begin{tikzpicture}[baseline=-.1cm]
	\coordinate (a) at (0,0);
	\coordinate (b) at (1.2,0);
	\draw ($(a)+(0,-.7)$) -- ($(a)+(0,.7)$);
	\draw ($(b)+(0,-.7)$) -- ($(b)+(0,.7)$);
	\draw ($(a)+(0,-.2)$) arc (180:0:.6cm);
	\draw ($(a)+(0,.2)$) arc (-180:0:.6cm);
	\ncircle{unshaded}{(a)}{.4}{180}{$R$}
	\ncircle{unshaded}{(b)}{.4}{0}{$R^{-1}$}
\end{tikzpicture}
=\,
\begin{tikzpicture}[baseline=.3cm]
	\draw (0,0) arc (180:0:.3cm);
	\draw (0,.8) arc (-180:0:.3cm);
\end{tikzpicture}\,,
$
\item
$
\begin{tikzpicture}[baseline=-.1cm]
	\coordinate (a) at (0,0);
	\draw ($(a)+(-.2,-.7)$) -- ($(a)+(-.2,.7)$);
	\draw ($(a)+(.2,.2)$) arc (180:0:.2cm) -- ($(a)+(.6,-.2)$) arc (0:-180:.2cm);
	\ncircle{unshaded}{(a)}{.4}{180}{$R$}
\end{tikzpicture}
=
a \,
\begin{tikzpicture}[baseline=-.1cm]
	\coordinate (a) at (0,0);
	\draw ($(a)+(0,-.7)$) -- ($(a)+(0,.7)$);
\end{tikzpicture}
$
and
$
\begin{tikzpicture}[baseline=-.1cm]
	\coordinate (a) at (0,0);
	\draw ($(a)+(-.2,-.7)$) -- ($(a)+(-.2,.7)$);
	\draw ($(a)+(.2,.2)$) arc (180:0:.2cm) -- ($(a)+(.6,-.2)$) arc (0:-180:.2cm);
	\ncircle{unshaded}{(a)}{.4}{90}{$R$}
\end{tikzpicture}
=
\sigma a^{-1} \,
\begin{tikzpicture}[baseline=-.1cm]
	\coordinate (a) at (0,0);
	\draw ($(a)+(0,-.7)$) -- ($(a)+(0,.7)$);
\end{tikzpicture}
$
for some $a\in\Complex\setminus\{0\}$,
\item
$R-\sigma^2 \cF^2(R)\in \TL_{2,+}$
(which implies $R^{-1}-\sigma^2 \cF^2(R^{-1})\in \TL_{2,+}$),
\item
($\sigma$-braid relation)
$R$ satisfies
$$
\sigma
\begin{tikzpicture}[baseline=.4cm]
	\coordinate (a) at (0,0);
	\coordinate (b) at (0,1);
	\coordinate (c) at (1,.5);
	\draw (b) -- (c) -- (a);
	\draw ($(a)+(0,-.7)$) -- ($(b)+(0,.7)$);
	\draw (a) -- ($(a)+(.4,-.7)$);
	\draw (b) -- ($(b)+(.4,.7)$);
	\draw ($(c)+(0,1.2)$) -- ($(c)+(0,-1.2)$);
	\ncircle{unshaded}{(a)}{.4}{180}{$R$}
	\ncircle{unshaded}{(b)}{.4}{180}{$R$}
	\ncircle{unshaded}{(c)}{.4}{115}{$R^{-1}$}
\end{tikzpicture}
=
\sigma^2
\begin{tikzpicture}[xscale=-1, baseline=.4cm]
	\coordinate (a) at (0,0);
	\coordinate (b) at (0,1);
	\coordinate (c) at (1,.5);
	\draw (b) -- (c) -- (a);
	\draw ($(a)+(0,-.7)$) -- ($(b)+(0,.7)$);
	\draw (a) -- ($(a)+(.4,-.7)$);
	\draw (b) -- ($(b)+(.4,.7)$);
	\draw ($(c)+(0,1.2)$) -- ($(c)+(0,-1.2)$);
	\ncircle{unshaded}{(a)}{.4}{55}{$R^{-1}$}
	\ncircle{unshaded}{(b)}{.4}{75}{$R^{-1}$}
	\ncircle{unshaded}{(c)}{.4}{0}{$R$}
\end{tikzpicture}
,$$ and
\item
(the quadratic condition) $R^2\in\TL_{2,+}\oplus \Complex R$.
\end{enumerate}
\end{defn}

Of course, an honest braiding satisfying the quadratic condition is a quadratic  (+1)-braiding. One of our examples has an honest braiding, but the other two have  quadratic $(\pm i)$-braidings.
Note that the condition $R\pm R^{-1} \in \TL_{2,+}$ implies $R$ is quadratic.
This holds for our examples.
Note that $R$ is a quadratic $(+1)$-braiding if and only if $iR$ is a quadratic $(-1)$-braiding.

\begin{remark}
Perhaps it would be interesting to work out the definition of a (not necessarily quadratic) $\sigma$-braiding; one would need all the variations of Reidemeister III, not just the $\sigma$-braid relation. It's unclear whether this would be motivated by the existence of examples, however!
Observe that an honest braiding would be a (not necessarily quadratic) $(+1)$-braiding (carefully note the positions of the $\star$ in the $\sigma$-braid relation to see this is equivalent to a Reidemeister III move).
Moreover, any honest braiding satisfying the HOMFLYPT skein relation \cite{MR776477} is a quadratic $(+1)$-braiding.
\end{remark}

\begin{thm}\label{thm:Evaluable}
Suppose the 2-box $R$ is a quadratic $\sigma$-braiding.
Then the planar algebra generated by $R$ is evaluable.
\end{thm}
There are two nice proofs. One is an adaptation of the skein template algorithm for evaluating the HOMFLYPT polynomial, the other an adaption of an argument used in analyzing Hecke algebras.

\begin{proof}[Proof 1]
We associate to a closed diagram $D$ in $R$ and $R^{-1}$ a link $L(D)$ by replacing each $R$ with a positive crossing and each $R^{-1}$ with a negative crossing.
If $L(D_1)$ is isotopic to $L(D_2)$, then $D_1 = D_2 + \text{ lower order terms}$; to see this, just follow the sequence of Reidemeister moves witnessing the isotopy by the corresponding relations for a quadratic $\sigma$-braiding.
The quadratic hypothesis says that if $L(D_1)$ differs from $L(D_2)$ by a crossing change, again  $D_1 = D_2 + \text{ lower order terms}$. Now we show that all closed diagrams are evaluable by inducting on the pair $(u(L(D)), \#(L(D)))$ (ordered lexicographically), where $u$ is the unknotting number and $\#$ is the number of crossings (equivalently the number of $R$ or $R^{-1}$s in $D$).
If the unknotting number is zero, there is an isotopy that reduces the number of crossings.
Otherwise, by the definition of unknotting number, there is some isotopy followed by a crossing change which decreases the unknotting number.
A diagram with $\#(L(D)) = 0$ is in Temperley-Lieb, providing the base case of the induction.
\end{proof}

\begin{proof}[Proof 2]
This time, we induct on the number of generators appearing in a diagram.

By an isotopy, we can arrange a closed diagram in $R$ and $R^{-1}$ as the trace of a `rectangular' diagram without any cups or caps, and with each generator having two strings up and two strings down. (This is the same as the argument in Alexander's theorem \cite{alexanders-theorem} that every link is the trace closure of a braid.) Such a diagram consists of $n$
parallel strings, and some number of generators bridging adjacent pairs of strings, as in the following example:
$$
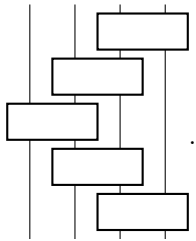
\begin{tikzpicture}[xscale=1.5,yscale=0.6,baseline=-.3cm]
	\coordinate (a) at (.8,2);
	\coordinate (b) at (.4,1);
	\coordinate (c) at (0,0);
	\coordinate (d) at (.4,-1);
	\coordinate (e) at (.8,-2);
	\draw ($(a)+(-1,.6)$)--($(e)+(-1,-.6)$);
	\draw ($(a)+(-.6,.6)$)--($(e)+(-.6,-.6)$);
	\draw ($(a)+(-.2,.6)$)--($(e)+(-.2,-.6)$);
	\draw ($(a)+(.2,.6)$)--($(e)+(.2,-.6)$);
	\nbox{unshaded}{(a)}{.4}{0}{0}{}
	\nbox{unshaded}{(b)}{.4}{0}{0}{}
	\nbox{unshaded}{(c)}{.4}{0}{0}{}
	\nbox{unshaded}{(d)}{.4}{0}{0}{}
	\nbox{unshaded}{(e)}{.4}{0}{0}{}
\end{tikzpicture}.
$$

Now, we argue that it is possible to use the relations to arrange that there is at most one generator connected to the rightmost string. If there are at least two, start bringing two adjacent generators on the rightmost string together. If there are no intervening generators on the rightmost-but-one string, we can bring them together and use the quadratic relation to replace the pair of generators with a single one (and lower order terms,
which are handled via our induction). If there is exactly one intervening
generator, we can use condition (5) for a quadratic $\sigma$-braiding (the
analogue of Reidemeister III) to reduce the number of generators connected to
the rightmost string. If there are at least two intervening generators, we
first apply argument of this paragraph again to those generators, replacing them with a single generator, and then apply
the previous sentence!
\end{proof}

Note that the evaluation of a diagram depends on $\sigma$, $a$, and also the implicit structure constants in relations (3) and (5) above.

\begin{remark}
By unpublished work of Bisch and Jones, and independently Wenzl,  all quadratic $(+1)$-braided subfactor planar algebras come from BMW algebras.
This also follows from Liu's classification work \cite{PASmallDimPartIV}.
\end{remark}

%%%%%%%%%%%%%%%%%%%%%%%%%%%%%%%%%%%%%%%%%%%%%%%%%%%%%%%%%%%%%%%%%%%%%
\subsection{Existence and Uniqueness of \texorpdfstring{$\cS,\cS'$}{S, S'}}\label{sec:ShurikenExistence}

Let $\Gamma$ be the underlying graph (without dual data) of the principal graph of $\cS,\cS'$.
Let $\cG_\bullet$ denote the bipartite graph planar algebra of $\Gamma$ \cite{MR1865703}.

The following two lemmas are the first ingredients of the existence and uniqueness of $\cS$ and $\cS'$.
We defer the technical proof to Appendix \ref{sec:Appendix}.

\begin{lem}\label{lem:OneParameterPlus}
Up to a sign, there is exactly a 1-parameter family of elements $A=\pm A(\lambda)\in \cG_{2,+}$ where $\lambda \in \mathbb T$ which satisfy
\begin{enumerate}[(1)]
\item
$A^*=A$ and $\cF^2(A)=A$,
\item
$A$ is uncappable, and
\item
$A^2=\jw{2}$ and $\cF(A)^2=\jw{2}$.
\end{enumerate}
\end{lem}

\begin{lem}\label{lem:OneParameterMinus}
Up to a sign, there are exactly two 1-parameter families of elements $B=\pm B(\lambda,\varepsilon)\in \cG_{2,+}$  where $\lambda \in \mathbb T$ and $\varepsilon \in \{-1,1\}$ which satisfy
\begin{enumerate}[(1)]
\item
$B^*=B$ and $\cF^2(B)=-B$,
\item
$B$ is uncappable, and
\item
$B^2=\jw{2}$ and $\check{B}^2=(-i\cF(B))^2=\jw{2}$.
\end{enumerate}
\end{lem}

Using the elements found in Lemmas \ref{lem:OneParameterPlus} and \ref{lem:OneParameterMinus}, the next 2 propositions follow from the classification results in \cite{PASmallDimPartIV}, or by a straightforward, but tedious calculation in the graph planar algebra $\cG_\bullet$.
We also check these computations directly in $\cG_\bullet$ in the {\tt Mathematica} notebook {\tt BraidRelationsForS.nb} bundled with the {\tt arXiv} source of this article.

\begin{prop}\label{prop:YangBaxterPlus}
The element
$$
R=R(\lambda)= \left(\frac{\sqrt{2}}{2}i\right) (\id_{2,+}-(1+\sqrt{2}) e_1)  + \left(\frac{\sqrt{2}}{2}\right)A(\lambda)
$$
is a quadratic $(+1)$-braiding such that the following hold independently of $\lambda$:
\begin{itemize}
\item
$a=i$
\item
$\cF^2(R)=R$, and 
\item
$R-R^{-1}=(i \sqrt{2})\id_{2,+}- i (2 + \sqrt{2})e_1$.
\end{itemize}
\end{prop}

\begin{prop}\label{prop:YangBaxterMinus}
The element
$$
S=S(\lambda,\varepsilon)=\left(\frac{\varepsilon i}{\sqrt{4+2\sqrt{2}}}\right) \id_{2,+} + \left(\frac{1+\sqrt{2}}{\sqrt{4+2\sqrt{2}}}\right) (e_1 + B(\lambda, \varepsilon))
$$
is a quadratic $(\varepsilon i)$-braiding such that the following hold independently of $\lambda$:
\begin{itemize}
\item
$a_\varepsilon=\exp\left(\varepsilon \frac{3\pi i}{8}\right)$
\item
$S-\sigma^2\cF^2(S)=S+\cF^2(S)=\sigma\left(\sqrt{2-\sqrt{2}}\right)\id_{2,+}+ \left(\sqrt{2+\sqrt{2}}\right)e_1$, and
\item
$S-S^{-1}= \varepsilon i \left(\sqrt{2-\sqrt{2}}\right)\id_{2,+}$.
\end{itemize}
\end{prop}

With the above lemmas and propositions, we now prove the main results of this section.

\begin{thm}\label{thm:ExistUnique}
There exists a unique subfactor planar algebra with principal graphs $\cS$.
\end{thm}
\begin{proof}
By Lemma \ref{lem:OneParameterPlus}, there is a 1-parameter family of elements $A(\lambda)$ in $\cG_{2,+}$ (the graph planar algebra), each of which gives a $(+1)$-braiding $R(\lambda)$ which satisfies the relations in Proposition \ref{prop:YangBaxterPlus}, independent of $\lambda$.
By Theorem \ref{thm:Evaluable}, each $R(\lambda)$ generates an isomorphic evaluable 1-supertransitive subfactor planar algebra.
By Theorem \ref{thm:MainHelper} and the results of Sections \ref{sec:Dart} and \ref{sec:Kangaroo}, the only possibility for the principal graph pair is $\cS$. This establishes existence.

To prove uniqueness, we note that any subfactor planar algebra with principal graphs $\cS$ has an uncappable self-adjoint low-weight rotational eigenvector $A$ with $\omega_A=1$ satisfying $A^2=\jw{2}$ and $\cF(A)^2=\jw{2}$.
By the graph planar algebra embedding theorem \cite{MR2812459}, there must be a  solution for these equations in the graph planar algebra $\cG_\bullet$.
These solutions are, up to sign, precisely the 1-parameter family found in Lemma \ref{lem:OneParameterPlus}. If we have $A = \pm A(\lambda)$, we then consider the $(+1)$-braiding
\begin{align*}
R & = \left(\frac{\sqrt{2}}{2}i\right) (\id_{2,+}-(1+\sqrt{2}) e_1)  \pm \left(\frac{\sqrt{2}}{2}\right)A \\
    &=  \left(\frac{\sqrt{2}}{2}i\right) (\id_{2,+}-(1+\sqrt{2}) e_1)  + \left(\frac{\sqrt{2}}{2}\right)A(\lambda)
\end{align*}
as in Proposition \ref{prop:YangBaxterPlus}. The planar algebra generated by this $R$ is independent of the choice of $\lambda$ and $\pm$ in $A=\pm A(\lambda)$, giving uniqueness. (Observe that although the subfactor planar algebra is unique, it does not embed uniquely in the graph planar algebra, giving the family $\pm A(\lambda)$. This is a because of the gauge action, c.f. \cite[Section 1.4]{1205.2742}.)

By uniqueness, the planar algebra $A$ generates must be self-dual. We can also check this explicitly by a straightforward calculation in $\cG_\bullet$. We have 
$$
\check{R} = \check{R}(\lambda)=  \left(\frac{\sqrt{2}}{2}i\right) (\id_{2,-}-(1+\sqrt{2}) e_1)  + \left(\frac{\sqrt{2}}{2}\right)\cF(A(\lambda))
$$
is a quadratic $(+1)$-braiding satisfying the same relations as $R$, so the map $R\leftrightarrow \check{R}$, or equivalently $A\leftrightarrow \check{A}=\cF(A)$, induces an isomorphism with the dual planar algebra.
We also check these computations directly in $\cG_\bullet$ in the {\tt Mathematica} notebook {\tt BraidRelationsForS.nb} bundled with the {\tt arXiv} source of this article.
\end{proof}

\begin{remark}
There is an easier construction of $\cS$.
It is the reduced subfactor of $A_7$ at $\jw{2}$.
We don't see, however, a corresponding easier proof of uniqueness.
\end{remark}

\begin{thm}\label{thm:ExistUniquePrime}
There exist exactly two distinct self-dual subfactor planar algebras with principal graphs $\cS'$ which are complex conjugate to each other.
\end{thm}
\begin{proof}
The proof for existence and uniqueness for a fixed choice of $\varepsilon$ is identical to the proof of existence and uniqueness from Theorem \ref{thm:ExistUnique}, exchanging $A(\lambda)$ with $B(\lambda,\varepsilon)$ and $R(\lambda)$ with $S(\lambda,\varepsilon)$ and $\cS$ with $\cS'$.

We note that the $B(\lambda,+1)$ and the $B(\lambda,-1)$ generate non-isomorphic subfactor planar algebras.
In the planar algebra generated by $B(\lambda,\varepsilon)$, we look for the low weight rotational eigenvectors $B$, satisfying
\begin{enumerate}[(1)]
\item
$B^*=B$ and $\cF^2(B)=-B$,
\item
$B$ is uncappable, and
\item
$B^2=\jw{2}$ and $\check{B}^2=(-i\cF(B))^2=\jw{2}$,
\end{enumerate}
and find $\pm B(\lambda,\varepsilon)$. 
We next evaluate the octahedron in $B$
$$
\Tr \left(
\begin{tikzpicture}[baseline=-.3cm]
	\coordinate (a) at (0,0);
	\coordinate (b) at (0,1);
	\coordinate (c) at (1,.5);
	\coordinate (d) at (1,-.5);
	\coordinate (e) at (1,-1.5);
	\coordinate (f) at (0,-1);
	\draw (b) -- (c) -- (a) -- (d) -- (f) -- (e);
	\draw ($(f)+(0,-1.2)$) -- ($(b)+(0,.7)$);
	\draw (e) -- ($(e)+(-.4,-.7)$);
	\draw (b) -- ($(b)+(.4,.7)$);
	\draw ($(c)+(0,1.2)$) -- ($(e)+(0,-.7)$);
	\ncircle{unshaded}{(a)}{.4}{180}{$B$}
	\ncircle{unshaded}{(b)}{.4}{180}{$B$}
	\ncircle{unshaded}{(c)}{.4}{115}{$B$}
	\ncircle{unshaded}{(e)}{.4}{115}{$B$}
	\ncircle{unshaded}{(d)}{.4}{115}{$B$}
	\ncircle{unshaded}{(f)}{.4}{180}{$B$}
\end{tikzpicture}
\right)
=
\varepsilon 16(1-\sqrt{2}),
$$
which is independent of the choice $\pm$ because it has an even number of vertices. 
Thus $\varepsilon$ is an invariant of the planar algebra.

Each of these subfactors is self-dual.
One can check by a straightforward calculation in $\cG_\bullet$ that
$$
T(\lambda,\varepsilon)= \left(\frac{\varepsilon i}{\sqrt{4+2\sqrt{2}}}\right) \id_{2,-} + \left(\frac{1+\sqrt{2}}{\sqrt{4+2\sqrt{2}}}\right) (e_1 - i\cF(B(\lambda, \varepsilon)))
$$
is a quadratic $(\varepsilon i)$-braiding satisfying the same relations as $S(\lambda,\varepsilon)$, so the map $S(\lambda,\varepsilon)\leftrightarrow T(\lambda,\varepsilon)$, or equivalently $B\leftrightarrow \check{B}=-i\cF(B)$, induces an isomorphism with the dual planar algebra.
To see these subfactors are complex conjugate to each other, one can check by a straightforward calculation in $\cG_\bullet$ that $B(\lambda,\varepsilon)$ is the entry-wise complex conjugate of $B(\overline{\lambda},-\varepsilon)$.
We also check these computations directly in $\cG_\bullet$ in the {\tt Mathematica} notebook {\tt BraidRelationsForS.nb} bundled with the {\tt arXiv} source of this article.
\end{proof}

\begin{remark}
The self-dualities in the proofs of Theorems \ref{thm:ExistUnique} and \ref{thm:ExistUniquePrime} are symmetric in the sense of \cite[Section 5.1]{1208.3637}.
Hence these planar algebras are really unshaded, and they give unitary $\bbZ/2\bbZ$-graded fusion categories.
\end{remark}

%%%%%%%%%%%%%%%%%%%%%%%%%%%%%%%%%%%%%%%%%%%%%%%%%%%%%%%%%%%%%%%%%%%%%%%%%%%%%%%%%%%%%%%%%%%%%%%%%%%%%%%%%%%%%%%%%%%%%%%%%%%%%%%%%%%%%%%%%%%%%%%%%%%%%%%%%%%%%%%%%%%%%%%%%%%%%%%%%%%%%%%%%%%%%%%%%%%%%%%%%%%%%%%
\section{Eliminating \texorpdfstring{$\cD,\cD'$}{D, D'}}\label{sec:Dart}

In this section we show that $\cS,\cS'$ are the only principal graphs which are extensions of the graphs $\cD$ or $\cD'$ at most index $6 \frac{1}{5}$, where
\begin{align*}
\cD & = \left(\bigraph{bwd1v1p1v1x0p1x1p0x1duals1v1x2},\bigraph{bwd1v1p1v1x0p1x1p0x1duals1v1x2} \right) &
\cD' & = \left(\bigraph{bwd1v1p1v1x0p1x1p0x1duals1v2x1}, \bigraph{bwd1v1p1v1x0p1x1p0x1duals1v2x1} \right).
\end{align*}

We consider the cases $\cD,\cD'$ at the same time.
We give the following names to the vertices with depth at most 3 of $\Gamma_+$.
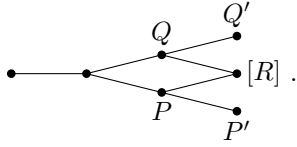
\begin{align*}
\begin{tikzpicture}[baseline=-.1cm]
\draw[fill] (0,0) circle (0.05);
\draw (0.,0.) -- (1.,0.);
\draw[fill] (1.,0.) circle (0.05);
\draw (1.,0.) -- (2.,-0.25);
\draw (1.,0.) -- (2.,0.25);
\draw[fill] (2.,-0.25) circle (0.05) node[below]{$P$};
\draw[fill] (2.,0.25) circle (0.05) node[above]{$Q$};
\draw (2.,-0.25) -- (3.,-0.5);
\draw (2.,-0.25) -- (3.,0.);
\draw (2.,0.25) -- (3.,0.);
\draw (2.,0.25) -- (3.,0.5);
\draw[fill] (3.,-0.5) circle (0.05) node[below]{$P'$};
\draw[fill] (3.,0.) circle (0.05) node[right]{$[R]$};
\draw[fill] (3.,0.5) circle (0.05) node[above]{$Q'$};
\end{tikzpicture}
.
\end{align*}

\begin{fact}\label{fact:Dimensions}
We have the following formulas for the dimensions of the vertices.
\begin{align*}
\dim(P)+\dim(Q)&=[3]\\
\dim([R])&=\frac{\dim(P)+\dim(Q)}{[2]}=\frac{[3]}{[2]}\\
\dim(P') &= [2]\dim(P)-\dim([R])-[2] \\
\dim(Q') &= [2]\dim(Q)-\dim([R])-[2].
\end{align*}
\end{fact}

We divide the extensions of $\cD,\cD'$ into three families based on which vertices at depth 3 connect to a vertex at depth 4.
\begin{defn}
\mbox{}
\begin{enumerate}
\item
An extension of $\cD$ or $\cD'$ is in $\cD_1$ if exactly one of $P',Q'$ connects to a vertex at depth 4, and $[R]$ and the other do not.
\item
An extension of $\cD$ or $\cD'$ is in $\cD_2$ if both $P',Q'$ connect to vertices at depth 4, and $[R]$ does not.
\item
An extension of $\cD$ or $\cD'$ is in $\cD_3$ if $[R]$ connects to a vertex at depth 4.
\end{enumerate}
\end{defn}

We rule out these families below index $6 \frac{1}{5}$ separately in the following subsections.

%%%%%%%%%%%%%%%%%%%%%%%%%%%%%%%%%%%%%%%%%
\subsection{Ruling out \texorpdfstring{$\cD_1$}{D_1}}
We now suppose that $(\Gamma_+,\Gamma_-)$ is in $\cD_1$. By symmetry, we may assume that only $Q'$ connects to vertices at depth 4.

\begin{lem}
$$
\displaystyle\dim(P')
=\frac{1+2[3]}{[2][3]}
=\frac{q \left(2 q^4+3 q^2+2\right)}{\left(q^2+1\right)\left(q^2-q+1\right) \left(q^2+q+1\right)}.
$$
\end{lem}
\begin{proof}
Since $\dim(P)=[2]\dim(P')$, it follows from Fact \ref{fact:Dimensions} that
$$
\dim(P')
=[2]^2\dim(P')-\dim([R])-[2]
\Longleftrightarrow
\dim(P')
=\frac{\dim([R])+[2]}{[3]}
=\frac{1+2[3]}{[2][3]}.
$$
\end{proof}

\begin{cor}
No graph pairs in $\cD_1$ are principal graphs of subfactors.
\end{cor}
\begin{proof}
Assume all vertices have dimension at least one.
First,
$$
\frac{\partial}{\partial q} \dim(P')
=
\frac{-2 q^{10}-5 q^8-4 q^6+4 q^4+5 q^2+2}{\left(q^6+2 q^4+2 q^2+1\right)^2},
$$
which is easily seen to be negative (e.g. by pairing up terms in the numerator, and using $q>1$).
Since $\dim(P') \approx 0.998428 $ at $q=1.64$, if $\dim(P') \geq 1$, we must have $q< 1.64$.
But since the norm of $\cD_1$ is $q+q^{-1}$ where $q\approx 1.76918$, we have $q>1.76$, a contradiction.
\end{proof}

%%%%%%%%%%%%%%%%%%%%%%%%%%%%%%%%%%%%%%%%%
\subsection{Ruling out \texorpdfstring{$\cD_2$}{D_2}}
We now suppose that $(\Gamma_+,\Gamma_-)$ is in $\cD_2$, so both $P'$ and $Q'$ connect to vertices at depth 4 of $\Gamma_+$.

\begin{lem}\label{lem:Automorphism}
Suppose a vertex $V$ of $\Gamma_+$ has dimension 1. Then $V$ is univalent.
\end{lem}
\begin{proof}
Since $V$ has dimension 1, $V\otimes \overline{V}\cong \id$ and $\overline{V}\otimes V\cong \id$.
The edges of $\Gamma_+$ correspond to tensoring with the standard bimodule $\sb{N}M_M$ or its dual $\sb{M}\overline{M}_N$.
Since $\cP_\bullet$ is 1-supertransitive, $M$ is irreducible, and thus $VM$ or $V\overline{M}$ (depending on whether $V$ occurs at an even or odd depth) is irreducible by Frobenius reciprocity.
\end{proof}

\begin{lem}\label{lem:Univalent}
When $\delta^2< 6.245$, no vertex at depth 4 of $\Gamma_+$ can have dimension 1.
\end{lem}
\begin{proof}
First, note that if $(q+q^{-1})^2=\delta^2=6.245$, then $q\approx 1.99867$.
Without loss of generality, suppose $P'$ attaches to a vertex of dimension 1, which must be univalent by Lemma \ref{lem:Automorphism}.
Then $\dim(P')=[2]$, and using Fact \ref{fact:Dimensions}, we calculate
\begin{align*}
\dim(P)&=\frac{2[2]+\frac{[3]}{[2]}}{[2]}=\frac{2+3[3]}{[2]^2}
\\
\dim(Q)&=[3]-\dim(P)=\frac{[5]-[3]-1}{[2]^2}
\\
\dim(Q')&=[2]\dim(Q)-[2]-\frac{[3]}{[2]}=\frac{[5]-3[3]-2}{[2]}=\frac{q^8-2 q^6-4 q^4-2 q^2+1}{q^3 \left(q^2+1\right)}.
\end{align*}
Now when $q<1.9988$, $\dim(Q')<\sqrt{2}$. Hence $\dim(Q')=1$, which means $Q'$ is univalent by Lemma \ref{lem:Automorphism}. But then $(\Gamma_+,\Gamma_-)\in \cD_1$, not $\cD_2$, a contradiction.
\end{proof}

\begin{lem}\label{lem:D2}
Let $P'',Q''$ be the sum of all vertices at depth $4$ connected to $P',Q'$ respectively. Then
$$
\dim(P'')+\dim(Q'')=[5]-3[3]-1=\frac{q^8-2 q^6-3 q^4-2 q^2+1}{q^4}.
$$
\end{lem}
\begin{proof}
Using Fact \ref{fact:Dimensions} and the formulas
\begin{align*}
\dim(P'') &= [2]\dim(P')-\dim(P) \\
\dim(Q'') &= [2]\dim(Q')-\dim(Q),
\end{align*}
it follows that
\begin{align*}
\dim(P'')+\dim(Q'')
&= [2](\dim(P')+\dim(Q'))-(\dim(P)+\dim(Q)) \\
&= [2]( [2](\dim(P)+\dim(Q))-2\dim([R])-2[2])-(\dim(P)+\dim(Q)) \\
&=[2]\left([2][3]-2\frac{[3]}{[2]}-2[2]\right)-[3] \\
&= [5] - 3[3] - 1.
\end{align*}
\end{proof}

\begin{prop}\label{prop:Only2Edges}
If $\delta^2\leq 6 \frac{1}{5}$, then there are exactly 2 edges in $\Gamma_+$ between depths 3 and 4.
\end{prop}
\begin{proof}
First, note that if $(q+q^{-1})^2=\delta^2=6 \frac{1}{5}$, then $q\approx 1.98661$.
Suppose there are $N$ edges in total from $P'$ and $Q'$ to vertices at depth 4. Then since each vertex has dimension at least 1,
$$
\dim(P'')+\dim(Q'')=[5]-3[3]-1=\frac{q^8-2 q^6-3 q^4-2 q^2+1}{q^4}\geq N.
$$
First, if $q<2.017$, then
$$
\frac{q^8-2 q^6-3 q^4-2 q^2+1}{q^4} < 5.
$$
Thus $N\leq 4$.
Second, if $q<1.993$, then
$$
\frac{q^8-2 q^6-3 q^4-2 q^2+1}{q^4} < 3+\sqrt{2}.
$$
Hence if $N=4$, all four vertices at depth 4 must have dimension exactly 1, contradicting Lemma \ref{lem:Univalent}.
Thus $N\leq 3$.
Third, if $q < 1.9867$, then
\begin{equation}\label{eqn:FinalInequality}
\frac{q^8-2 q^6-3 q^4-2 q^2+1}{q^4} < 3\sqrt{2}.
\end{equation}
Hence if $N=3$, at least one vertex at depth 4 has dimension 1, contradicting Lemma \ref{lem:Univalent}.
Thus $N=2$.
\end{proof}

\begin{remark}
In fact, all we needed was Inequality \eqref{eqn:FinalInequality}, since $3\sqrt{2}<3+\sqrt{2}<5$.
However, to classify 1-supertransitive subfactors to higher indices in the future, it may be helpful to have these more precise bounds at our disposal.
\end{remark}

\begin{prop}
No graph pairs in $\cD_2$ with $\delta^2\leq 6\frac{1}{5}$ are principal graphs of subfactors.
\end{prop}
\begin{proof}
Running the odometer to depth 5 on $\{\cD,\cD'\}$, there are no remaining weeds in $\cD_2$ which have exactly 2 edges in $\Gamma_+$ between depths 3 and 4.
See Figures \ref{fig:TikzTreeD2} and \ref{fig:TikzTreeD2Prime} for the output of running the odometer on $\cD$ and $\cD'$ respectively, and only keeping results which are in $\cD_2$ with exactly 2 edges between depths 3 and 4.
(Recall that the red background denotes an active weed, while a blue background denotes a cylinder.)
The vines and cylinders in $\cD_2$ either have a vertex with dimension less than one or have a vertex whose dimension is not an algebraic integer.
\begin{figure}[!ht]
\begin{tikzpicture}
\tikzset{grow=right,level distance=130pt}
\tikzset{every tree node/.style={draw,fill=white,rectangle,rounded corners,inner sep=2pt}}
\Tree
[.\node{$\!\!\begin{array}{c}\bigraph{bwd1v1p1v1x0p1x1p0x1duals1v1x2}\\\bigraph{bwd1v1p1v1x0p1x1p0x1duals1v1x2}\end{array}\!\!$};
	[.\node{$\!\!\begin{array}{c}\bigraph{bwd1v1p1v1x0p1x1p0x1v1x0x1duals1v1x2v1}\\\bigraph{bwd1v1p1v1x0p1x1p0x1v1x0x1duals1v1x2v1}\end{array}\!\!$};
		[.\node[fill=blue!30]{$\!\!\begin{array}{c}\bigraph{bwd1v1p1v1x0p1x1p0x1v1x0x1v1duals1v1x2v1}\\\bigraph{bwd1v1p1v1x0p1x1p0x1v1x0x1v1duals1v1x2v1}\end{array}\!\!$};]]
	[.\node[fill=blue!30]{$\!\!\begin{array}{c}\bigraph{bwd1v1p1v1x0p1x1p0x1v1x0x0p0x0x1duals1v1x2v1x2}\\\bigraph{bwd1v1p1v1x0p1x1p0x1v1x0x0p0x0x1duals1v1x2v1x2}\end{array}\!\!$};]]
\end{tikzpicture}
\caption{Results in $\cD_2$ with 2 edges between depths 3 and 4 of running the odometer on $\cD$ to depth 5.}
\label{fig:TikzTreeD2}
\end{figure}
\begin{figure}[!ht]
\begin{tikzpicture}
\tikzset{grow=right,level distance=130pt}
\tikzset{every tree node/.style={draw,fill=white,rectangle,rounded corners,inner sep=2pt}}
\Tree
[.\node{$\!\!\begin{array}{c}\bigraph{bwd1v1p1v1x0p1x1p0x1duals1v2x1}\\\bigraph{bwd1v1p1v1x0p1x1p0x1duals1v2x1}\end{array}\!\!$};
	[.\node{$\!\!\begin{array}{c}\bigraph{bwd1v1p1v1x0p1x1p0x1v1x0x1duals1v2x1v1}\\\bigraph{bwd1v1p1v1x0p1x1p0x1v1x0x1duals1v2x1v1}\end{array}\!\!$};
		[.\node[fill=blue!30]{$\!\!\begin{array}{c}\bigraph{bwd1v1p1v1x0p1x1p0x1v1x0x1v1duals1v2x1v1}\\\bigraph{bwd1v1p1v1x0p1x1p0x1v1x0x1v1duals1v2x1v1}\end{array}\!\!$};]]
	[.\node[fill=blue!30]{$\!\!\begin{array}{c}\bigraph{bwd1v1p1v1x0p1x1p0x1v1x0x0p0x0x1duals1v2x1v2x1}\\\bigraph{bwd1v1p1v1x0p1x1p0x1v1x0x0p0x0x1duals1v2x1v2x1}\end{array}\!\!$};]]
\end{tikzpicture}
\caption{Results in $\cD_2$ with 2 edges between depths 3 and 4 of running the odometer on $\cD'$ to depth 5.}
\label{fig:TikzTreeD2Prime}
\end{figure}
\end{proof}

%%%%%%%%%%%%%%%%%%%%%%%%%%%%%%%%%%%%%%%%%
\subsection{Ruling out \texorpdfstring{$\cD_3$}{D_3}}
We now suppose that $(\Gamma_+,\Gamma_-)$ is in $\cD_3$, so $[R]$ connects to a vertex at depth 4 of $\Gamma_+$.

\begin{prop}\label{prop:NoD3}
When $\delta^2\leq 6 \frac{1}{5}$, the only graph pairs in $\cD_3$ which are principal graphs of subfactors are $\cS$ and $\cS'$.
\end{prop}

\begin{proof}
Running the odometer on $\{\cD,\cD'\}$ to depth 6, the vines and cylinders in $\cD_3$ either have a vertex with dimension less than one or have a vertex whose dimension is not an algebraic integer, or are exactly $\cS,\cS'$. The weeds in $\cD_3$ are given in Figures \ref{fig:TikzTreeD3} and \ref{fig:TikzTreeD3Prime}.

\begin{figure}[!ht]
\resizebox{!}{10cm}{
\begin{tikzpicture}
\tikzset{grow=right,level distance=130pt}
\tikzset{every tree node/.style={draw,fill=white,rectangle,rounded corners,inner sep=2pt}}
\Tree
[.\node{$\!\!\begin{array}{c}\bigraph{bwd1v1p1v1x0p1x1p0x1duals1v1x2}\\\bigraph{bwd1v1p1v1x0p1x1p0x1duals1v1x2}\end{array}\!\!$};
	[.\node[fill=blue!30]{$\!\!\begin{array}{c}\bigraph{bwd1v1p1v1x0p1x1p0x1v0x1x0duals1v1x2v1}\\\bigraph{bwd1v1p1v1x0p1x1p0x1v0x1x0duals1v1x2v1}\end{array}\!\!$};]
	[.\node[fill=blue!30]{$\!\!\begin{array}{c}\bigraph{bwd1v1p1v1x0p1x1p0x1v1x0x0p0x1x0duals1v1x2v1x2}\\\bigraph{bwd1v1p1v1x0p1x1p0x1v1x0x0p0x1x0duals1v1x2v1x2}\end{array}\!\!$};]
	[.\node{$\!\!\begin{array}{c}\bigraph{bwd1v1p1v1x0p1x1p0x1v0x1x0p0x0x1p0x0x1duals1v1x2v1x3x2}\\\bigraph{bwd1v1p1v1x0p1x1p0x1v0x1x0p0x0x1p0x0x1duals1v1x2v1x3x2}\end{array}\!\!$};]
	[.\node{$\!\!\begin{array}{c}\bigraph{bwd1v1p1v1x0p1x1p0x1v1x0x0p1x0x0p0x1x0duals1v1x2v1x2x3}\\\bigraph{bwd1v1p1v1x0p1x1p0x1v1x0x0p1x0x0p0x1x0duals1v1x2v2x1x3}\end{array}\!\!$};
		[.\node{$\!\!\begin{array}{c}\bigraph{bwd1v1p1v1x0p1x1p0x1v1x0x0p1x0x0p0x1x0v1x0x0p1x0x0duals1v1x2v1x2x3}\\\bigraph{bwd1v1p1v1x0p1x1p0x1v1x0x0p1x0x0p0x1x0v1x0x0p0x1x0duals1v1x2v2x1x3}\end{array}\!\!$};
			[.\node[fill=red!30]{$\!\!\begin{array}{c}\bigraph{bwd1v1p1v1x0p1x1p0x1v1x0x0p1x0x0p0x1x0v1x0x0p0x1x0v1x1duals1v1x2v2x1x3v1}\\\bigraph{bwd1v1p1v1x0p1x1p0x1v1x0x0p1x0x0p0x1x0v1x0x0p1x0x0v1x0p0x1duals1v1x2v1x2x3v2x1}\end{array}\!\!$};]
			[.\node[fill=red!30]{$\!\!\begin{array}{c}\bigraph{bwd1v1p1v1x0p1x1p0x1v1x0x0p1x0x0p0x1x0v1x0x0p1x0x0v1x0p0x1duals1v1x2v1x2x3v2x1}\\\bigraph{bwd1v1p1v1x0p1x1p0x1v1x0x0p1x0x0p0x1x0v1x0x0p0x1x0v1x0p1x0p0x1p0x1duals1v1x2v2x1x3v1x3x2x4}\end{array}\!\!$};]]]
	[.\node[fill=blue!30]{$\!\!\begin{array}{c}\bigraph{bwd1v1p1v1x0p1x1p0x1v1x0x0p0x1x0p0x0x1duals1v1x2v1x2x3}\\\bigraph{bwd1v1p1v1x0p1x1p0x1v1x0x0p0x1x0p0x0x1duals1v1x2v1x2x3}\end{array}\!\!$};]
	[.\node{$\!\!\begin{array}{c}\bigraph{bwd1v1p1v1x0p1x1p0x1v0x1x0p0x0x1p0x0x1duals1v1x2v1x2x3}\\\bigraph{bwd1v1p1v1x0p1x1p0x1v0x1x0p0x0x1p0x0x1duals1v1x2v1x2x3}\end{array}\!\!$};
		[.\node{$\!\!\begin{array}{c}\bigraph{bwd1v1p1v1x0p1x1p0x1v0x1x0p0x0x1p0x0x1v0x1x0p0x0x1duals1v1x2v1x2x3}\\\bigraph{bwd1v1p1v1x0p1x1p0x1v0x1x0p0x0x1p0x0x1v0x0x1p0x0x1duals1v1x2v1x2x3}\end{array}\!\!$};
			[.\node[fill=red!30]{$\!\!\begin{array}{c}\bigraph{bwd1v1p1v1x0p1x1p0x1v0x1x0p0x0x1p0x0x1v0x1x0p0x0x1v1x1duals1v1x2v1x2x3v1}\\\bigraph{bwd1v1p1v1x0p1x1p0x1v0x1x0p0x0x1p0x0x1v0x0x1p0x0x1v1x0p0x1duals1v1x2v1x2x3v1x2}\end{array}\!\!$};]
			[.\node[fill=red!30]{$\!\!\begin{array}{c}\bigraph{bwd1v1p1v1x0p1x1p0x1v0x1x0p0x0x1p0x0x1v0x0x1p0x0x1v1x0duals1v1x2v1x2x3v1}\\\bigraph{bwd1v1p1v1x0p1x1p0x1v0x1x0p0x0x1p0x0x1v0x1x0p0x0x1v1x0p1x0p0x1duals1v1x2v1x2x3v1x3x2}\end{array}\!\!$};]
			[.\node[fill=red!30]{$\!\!\begin{array}{c}\bigraph{bwd1v1p1v1x0p1x1p0x1v0x1x0p0x0x1p0x0x1v0x0x1p0x0x1v1x0p1x0duals1v1x2v1x2x3v2x1}\\\bigraph{bwd1v1p1v1x0p1x1p0x1v0x1x0p0x0x1p0x0x1v0x1x0p0x0x1v1x0p1x0p1x0p0x1duals1v1x2v1x2x3v2x1x4x3}\end{array}\!\!$};]
			[.\node[fill=red!30]{$\!\!\begin{array}{c}\bigraph{bwd1v1p1v1x0p1x1p0x1v0x1x0p0x0x1p0x0x1v0x0x1p0x0x1v0x1p0x1duals1v1x2v1x2x3v2x1}\\\bigraph{bwd1v1p1v1x0p1x1p0x1v0x1x0p0x0x1p0x0x1v0x1x0p0x0x1v1x0p0x1p0x1p0x1duals1v1x2v1x2x3v2x1x3x4}\end{array}\!\!$};]
			[.\node[fill=red!30]{$\!\!\begin{array}{c}\bigraph{bwd1v1p1v1x0p1x1p0x1v0x1x0p0x0x1p0x0x1v0x0x1p0x0x1v1x0p0x1duals1v1x2v1x2x3v1x2}\\\bigraph{bwd1v1p1v1x0p1x1p0x1v0x1x0p0x0x1p0x0x1v0x1x0p0x0x1v1x0p1x0p0x1p0x1duals1v1x2v1x2x3v1x3x2x4}\end{array}\!\!$};]
			[.\node[fill=red!30]{$\!\!\begin{array}{c}\bigraph{bwd1v1p1v1x0p1x1p0x1v0x1x0p0x0x1p0x0x1v0x0x1p0x0x1v1x0p1x0duals1v1x2v1x2x3v1x2}\\\bigraph{bwd1v1p1v1x0p1x1p0x1v0x1x0p0x0x1p0x0x1v0x1x0p0x0x1v1x0p1x0p1x0p0x1duals1v1x2v1x2x3v2x1x4x3}\end{array}\!\!$};]
			[.\node[fill=red!30]{$\!\!\begin{array}{c}\bigraph{bwd1v1p1v1x0p1x1p0x1v0x1x0p0x0x1p0x0x1v0x0x1p0x0x1v0x1p0x1duals1v1x2v1x2x3v1x2}\\\bigraph{bwd1v1p1v1x0p1x1p0x1v0x1x0p0x0x1p0x0x1v0x1x0p0x0x1v1x0p0x1p0x1p0x1duals1v1x2v1x2x3v2x1x3x4}\end{array}\!\!$};]]
		[.\node{$\!\!\begin{array}{c}\bigraph{bwd1v1p1v1x0p1x1p0x1v0x1x0p0x0x1p0x0x1v0x1x0p0x1x0p0x0x1duals1v1x2v1x2x3}\\\bigraph{bwd1v1p1v1x0p1x1p0x1v0x1x0p0x0x1p0x0x1v0x1x0p0x0x1p0x0x1duals1v1x2v1x2x3}\end{array}\!\!$};]]
	[.\node{$\!\!\begin{array}{c}\bigraph{bwd1v1p1v1x0p1x1p0x1v1x0x0p0x1x0p0x0x1p0x0x1duals1v1x2v1x2x4x3}\\\bigraph{bwd1v1p1v1x0p1x1p0x1v1x0x0p0x1x0p0x0x1p0x0x1duals1v1x2v1x2x4x3}\end{array}\!\!$};]
	[.\node{$\!\!\begin{array}{c}\bigraph{bwd1v1p1v1x0p1x1p0x1v1x0x0p0x1x0p0x0x1p0x0x1duals1v1x2v1x2x3x4}\\\bigraph{bwd1v1p1v1x0p1x1p0x1v1x0x0p0x1x0p0x0x1p0x0x1duals1v1x2v1x2x4x3}\end{array}\!\!$};]
	[.\node{$\!\!\begin{array}{c}\bigraph{bwd1v1p1v1x0p1x1p0x1v1x0x0p1x0x0p0x1x0p0x0x1duals1v1x2v1x2x3x4}\\\bigraph{bwd1v1p1v1x0p1x1p0x1v1x0x0p1x0x0p0x1x0p0x0x1duals1v1x2v1x2x3x4}\end{array}\!\!$};]]
\end{tikzpicture}
}
\caption{Results in $\cD_3$ of running the odometer on $\cD$ to depth 6.}
\label{fig:TikzTreeD3}
\end{figure}

\begin{figure}[!ht]
\begin{tikzpicture}
\tikzset{grow=right,level distance=130pt}
\tikzset{every tree node/.style={draw,fill=white,rectangle,rounded corners,inner sep=2pt}}
\Tree
[.\node{$\!\!\begin{array}{c}\bigraph{bwd1v1p1v1x0p1x1p0x1duals1v2x1}\\\bigraph{bwd1v1p1v1x0p1x1p0x1duals1v2x1}\end{array}\!\!$};
	[.\node[fill=blue!30]{$\!\!\begin{array}{c}\bigraph{bwd1v1p1v1x0p1x1p0x1v0x1x0duals1v2x1v1}\\\bigraph{bwd1v1p1v1x0p1x1p0x1v0x1x0duals1v2x1v1}\end{array}\!\!$};]
	[.\node[fill=blue!30]{$\!\!\begin{array}{c}\bigraph{bwd1v1p1v1x0p1x1p0x1v1x0x0p0x1x0p0x0x1duals1v2x1v3x2x1}\\\bigraph{bwd1v1p1v1x0p1x1p0x1v1x0x0p0x1x0p0x0x1duals1v2x1v3x2x1}\end{array}\!\!$};]]
\end{tikzpicture}
\caption{Results in $\cD_3$ of running the odometer on $\cD'$ to depth 6.}
\label{fig:TikzTreeD3Prime}
\end{figure}

For all the remaining weeds, there is a univalent vertex at depth 4 \underline{not} connected to $[R]$ on one of $\Gamma_\pm$.
Its relative dimension is given by
$$
d(q)=\frac{q^{16}-q^{14}-5 q^{12}-6 q^{10}+6 q^6+5 q^4+q^2-1}{q^2 \left(q^4-1\right) \left(q^4+q^2+1\right)^2}.
$$
Now $d(q)<1$ until $q\approx 2.03228$, which corresponds to index $6.37228$.
\end{proof}

%%%%%%%%%%%%%%%%%%%%%%%%%%%%%%%%%%%%%%%%%%%%%%%%%%%%%%%%%%%%%%%%%%%%%%%%%%%%%%%%%%%%%%%%%%%%%%%%%%%%%%%%%%%%%%%%%%%%%%%%%%%%%%%%%%%%%%%%%%%%%%%%%%%%%%%%%%%%%%%%%%%%%%%%%%%%%%%%%%%%%%%%%%%%%%%%%%%%%%%%%%%%%%%
\section{Eliminating \texorpdfstring{$\cK,\cK'$}{K, K'}}\label{sec:Kangaroo}

In this section we show that there are no principal graphs which are extensions of the graphs $\cK$ or $\cK'$ with index at most $6 \frac{1}{5}$, where
\begin{align*}
\cK & = \left(\bigraph{bwd1v1p1v1x1p0x1p0x1duals1v1x2},\bigraph{bwd1v1p1v1x1p0x1p0x1duals1v1x2}\right) &
\cK' & = \left(\bigraph{bwd1v1p1v1x0p1x1p0x1duals1v1x2},\bigraph{bwd1v1p1v0x1p1x1p0x1duals1v1x2}\right).
\end{align*}

\begin{prop}\label{prop:NoK}
When $\delta^2\leq 6 \frac{1}{5}$, no extensions of $\cK,\cK'$ are principal graphs of subfactors.
\end{prop}
\begin{proof}
First, we run the odometer to depth 5 on $\cK$, removing the cylinders along the way.
The output is given in Figure \ref{fig:TikzTreeK}.

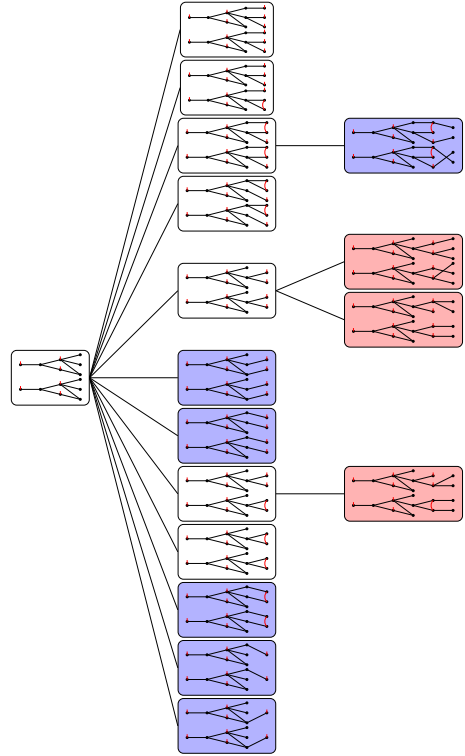
\begin{figure}[!ht]
\resizebox{!}{10cm}{
\begin{tikzpicture}
\tikzset{grow=right,level distance=130pt}
\tikzset{every tree node/.style={draw,fill=white,rectangle,rounded corners,inner sep=2pt}}
\Tree
[.\node{$\!\!\begin{array}{c}\bigraph{bwd1v1p1v1x1p0x1p0x1duals1v1x2}\\\bigraph{bwd1v1p1v1x1p0x1p0x1duals1v1x2}\end{array}\!\!$};
	[.\node[fill=blue!30]{$\!\!\begin{array}{c}\bigraph{bwd1v1p1v1x1p0x1p0x1v1x0x0duals1v1x2v1}\\\bigraph{bwd1v1p1v1x1p0x1p0x1v1x0x0duals1v1x2v1}\end{array}\!\!$};]
	[.\node[fill=blue!30]{$\!\!\begin{array}{c}\bigraph{bwd1v1p1v1x1p0x1p0x1v0x0x1duals1v1x2v1}\\\bigraph{bwd1v1p1v1x1p0x1p0x1v0x0x1duals1v1x2v1}\end{array}\!\!$};]
	[.\node[fill=blue!30]{$\!\!\begin{array}{c}\bigraph{bwd1v1p1v1x1p0x1p0x1v0x1x0p0x0x1duals1v1x2v2x1}\\\bigraph{bwd1v1p1v1x1p0x1p0x1v0x1x0p0x0x1duals1v1x2v2x1}\end{array}\!\!$};]
	[.\node{$\!\!\begin{array}{c}\bigraph{bwd1v1p1v1x1p0x1p0x1v0x1x0p0x1x0duals1v1x2v2x1}\\\bigraph{bwd1v1p1v1x1p0x1p0x1v0x1x0p0x1x0duals1v1x2v2x1}\end{array}\!\!$};]
	[.\node{$\!\!\begin{array}{c}\bigraph{bwd1v1p1v1x1p0x1p0x1v0x1x0p0x1x0duals1v1x2v1x2}\\\bigraph{bwd1v1p1v1x1p0x1p0x1v0x1x0p0x1x0duals1v1x2v2x1}\end{array}\!\!$};
		[.\node[fill=red!30]{$\!\!\begin{array}{c}\bigraph{bwd1v1p1v1x1p0x1p0x1v0x1x0p0x1x0v1x0p1x0duals1v1x2v1x2}\\\bigraph{bwd1v1p1v1x1p0x1p0x1v0x1x0p0x1x0v1x0p0x1duals1v1x2v2x1}\end{array}\!\!$};]]
	[.\node[fill=blue!30]{$\!\!\begin{array}{c}\bigraph{bwd1v1p1v1x1p0x1p0x1v0x1x0p0x0x1duals1v1x2v1x2}\\\bigraph{bwd1v1p1v1x1p0x1p0x1v0x1x0p0x0x1duals1v1x2v1x2}\end{array}\!\!$};]
	[.\node[fill=blue!30]{$\!\!\begin{array}{c}\bigraph{bwd1v1p1v1x1p0x1p0x1v1x0x0p0x1x0duals1v1x2v1x2}\\\bigraph{bwd1v1p1v1x1p0x1p0x1v1x0x0p0x1x0duals1v1x2v1x2}\end{array}\!\!$};]
	[.\node{$\!\!\begin{array}{c}\bigraph{bwd1v1p1v1x1p0x1p0x1v0x1x0p0x1x0duals1v1x2v1x2}\\\bigraph{bwd1v1p1v1x1p0x1p0x1v0x1x0p0x1x0duals1v1x2v1x2}\end{array}\!\!$};
		[.\node[fill=red!30]{$\!\!\begin{array}{c}\bigraph{bwd1v1p1v1x1p0x1p0x1v0x1x0p0x1x0v0x1p0x1duals1v1x2v1x2}\\\bigraph{bwd1v1p1v1x1p0x1p0x1v0x1x0p0x1x0v1x0p0x1duals1v1x2v1x2}\end{array}\!\!$};]
		[.\node[fill=red!30]{$\!\!\begin{array}{c}\bigraph{bwd1v1p1v1x1p0x1p0x1v0x1x0p0x1x0v1x0p1x0p0x1duals1v1x2v1x2}\\\bigraph{bwd1v1p1v1x1p0x1p0x1v0x1x0p0x1x0v1x0p0x1p1x0duals1v1x2v1x2}\end{array}\!\!$};]]
	[.\node{$\!\!\begin{array}{c}\bigraph{bwd1v1p1v1x1p0x1p0x1v0x1x0p0x0x1p0x0x1duals1v1x2v1x3x2}\\\bigraph{bwd1v1p1v1x1p0x1p0x1v0x1x0p0x0x1p0x0x1duals1v1x2v1x3x2}\end{array}\!\!$};]
	[.\node{$\!\!\begin{array}{c}\bigraph{bwd1v1p1v1x1p0x1p0x1v0x1x0p0x1x0p0x0x1duals1v1x2v1x3x2}\\\bigraph{bwd1v1p1v1x1p0x1p0x1v0x1x0p0x1x0p0x0x1duals1v1x2v1x3x2}\end{array}\!\!$};
		[.\node[fill=blue!30]{$\!\!\begin{array}{c}\bigraph{bwd1v1p1v1x1p0x1p0x1v0x1x0p0x1x0p0x0x1v1x0x0p0x0x1duals1v1x2v1x3x2}\\\bigraph{bwd1v1p1v1x1p0x1p0x1v0x1x0p0x1x0p0x0x1v0x0x1p1x0x0duals1v1x2v1x3x2}\end{array}\!\!$};]]
	[.\node{$\!\!\begin{array}{c}\bigraph{bwd1v1p1v1x1p0x1p0x1v0x1x0p0x1x0p0x0x1duals1v1x2v1x2x3}\\\bigraph{bwd1v1p1v1x1p0x1p0x1v0x1x0p0x1x0p0x0x1duals1v1x2v2x1x3}\end{array}\!\!$};]
	[.\node{$\!\!\begin{array}{c}\bigraph{bwd1v1p1v1x1p0x1p0x1v0x1x0p0x1x0p0x0x1duals1v1x2v1x2x3}\\\bigraph{bwd1v1p1v1x1p0x1p0x1v0x1x0p0x1x0p0x0x1duals1v1x2v1x2x3}\end{array}\!\!$};]]
\end{tikzpicture}
}
\caption{Results  of running the odometer on $\cK$ to depth 5.}
\label{fig:TikzTreeK}
\end{figure}

For the weeds which are extensions of $\cK$ which are not cylinders, the top vertex at depth 4 has relative dimension
$$
d(q)=\frac{q^{10}-1}{q \left(q^2-1\right) \left(q^2+1\right)^3}.
$$
Note that
\begin{itemize}
\item
if $(q+q^{-1})^2=\delta^2=3+\sqrt{5}$, then $q\approx 1.70002$, and
\item
if $(q+q^{-1})^2=\delta^2=6 \frac{1}{5}$, then $q\approx 1.98661$.
\end{itemize}
However, for $1.636 < q< 2.047$, $1< d(q) < \sqrt{2}$, which is impossible.

Note that the $\Gamma_+$ of $\cK'$ is the same as the $\Gamma_+$ of $\cD$.
Hence all the results of Section \ref{sec:Dart} which only relied on $\Gamma_+$ apply to these weeds.
In particular, Proposition \ref{prop:Only2Edges} applies, so there must be exactly 2 edges between depths 3 and 4 on any extension.

Thus we run the odometer to depth 6 on $\cK'$, removing cylinders along the way, and keeping only the extensions of $\cK'$ for which there are exactly 2 edges between depths 3 and 4 on the graph which truncates to $\Gamma_+$ of $\cD$.
The output is given in Figure \ref{fig:TikzTreeKPrime}.

\begin{figure}[!ht]
\resizebox{!}{5cm}{
\begin{tikzpicture}
\tikzset{grow=right,level distance=130pt}
\tikzset{every tree node/.style={draw,fill=white,rectangle,rounded corners,inner sep=2pt}}
\Tree
[.\node{$\!\!\begin{array}{c}\bigraph{bwd1v1p1v1x0p1x1p0x1duals1v1x2}\\\bigraph{bwd1v1p1v0x1p1x1p0x1duals1v1x2}\end{array}\!\!$};
	[.\node{$\!\!\begin{array}{c}\bigraph{bwd1v1p1v1x0p1x1p0x1v1x0x1duals1v1x2v1}\\\bigraph{bwd1v1p1v0x1p1x1p0x1v1x0x0p0x0x1duals1v1x2v1x2}\end{array}\!\!$};
		[.\node{$\!\!\begin{array}{c}\bigraph{bwd1v1p1v1x0p1x1p0x1v1x0x1v1duals1v1x2v1}\\\bigraph{bwd1v1p1v0x1p1x1p0x1v1x0x0p0x0x1v1x1duals1v1x2v1x2}\end{array}\!\!$};
			[.\node[fill=red!30]{$\!\!\begin{array}{c}\bigraph{bwd1v1p1v0x1p1x1p0x1v1x0x0p0x0x1v1x1v1duals1v1x2v1x2v1}\\\bigraph{bwd1v1p1v1x0p1x1p0x1v1x0x1v1v1p1duals1v1x2v1v2x1}\end{array}\!\!$};]
			[.\node[fill=red!30]{$\!\!\begin{array}{c}\bigraph{bwd1v1p1v0x1p1x1p0x1v1x0x0p0x0x1v1x1v1duals1v1x2v1x2v1}\\\bigraph{bwd1v1p1v1x0p1x1p0x1v1x0x1v1v1p1duals1v1x2v1v1x2}\end{array}\!\!$};]]]
	[.\node[fill=blue!30]{$\!\!\begin{array}{c}\bigraph{bwd1v1p1v0x1p1x1p0x1v0x1x0duals1v1x2v1}\\\bigraph{bwd1v1p1v1x0p1x1p0x1v1x0x0p0x1x0p0x0x1duals1v1x2v3x2x1}\end{array}\!\!$};]
	[.\node{$\!\!\begin{array}{c}\bigraph{bwd1v1p1v0x1p1x1p0x1v1x0x0p0x1x0duals1v1x2v1x2}\\\bigraph{bwd1v1p1v1x0p1x1p0x1v1x0x0p1x0x0p0x1x0p0x0x1duals1v1x2v1x4x3x2}\end{array}\!\!$};
		[.\node[fill=blue!30]{$\!\!\begin{array}{c}\bigraph{bwd1v1p1v0x1p1x1p0x1v1x0x0p0x1x0v1x0duals1v1x2v1x2}\\\bigraph{bwd1v1p1v1x0p1x1p0x1v1x0x0p1x0x0p0x1x0p0x0x1v0x0x0x1duals1v1x2v1x4x3x2}\end{array}\!\!$};]
		[.\node{$\!\!\begin{array}{c}\bigraph{bwd1v1p1v0x1p1x1p0x1v1x0x0p0x1x0v1x0p1x0duals1v1x2v1x2}\\\bigraph{bwd1v1p1v1x0p1x1p0x1v1x0x0p1x0x0p0x1x0p0x0x1v1x0x0x0p0x0x0x1duals1v1x2v1x4x3x2}\end{array}\!\!$};]]]
\end{tikzpicture}
}
\caption{Results of running the odometer on $\cK'$ to depth 6.}
\label{fig:TikzTreeKPrime}
\end{figure}

Now both of the two remaining weeds have the same $\Gamma_+$.
The vertex at depth 6 of this $\Gamma_+$ has relative dimension
$$
\frac{q^{18}-3 q^{16}-3 q^{14}-q^{12}+2 q^{10}-2 q^8+q^6+3 q^4+3 q^2-1}{2 q^6 \left(q^2-1\right) \left(q^2+1\right)^2},
$$
which is less than 1 for $q<1.987$.

The vines and cylinders from both $\cK$ and $\cK'$ all either have a vertex with dimension less than 1, or have a vertex whose dimension is not an algebraic integer.
\end{proof}

%%%%%%%%%%%%%%%%%%%%%%%%%%%%%%%%%%%%%%%%%%%%%%%%%%%%%%%%%%%%%%%%%%%%%
%%%%%%%%%%%%%%%%%%%%%%%%%%%%%%%%%%%%%%%%%%%%%%%%%%%%%%%%%%%%%%%%%%%%%
%%%%%%%%%%%%%%%%%%%%%%%%%%%%%%%%%%%%%%%%%%%%%%%%%%%%%%%%%%%%%%%%%%%%%
\appendix
\section{The proofs of Lemmas \ref{lem:OneParameterPlus} and \ref{lem:OneParameterMinus}}\label{sec:Appendix}

Let $\Gamma$ be the underlying graph of the principal graphs of both $\cS,\cS'$.
We label the vertices of $\Gamma$ as follows:
$$
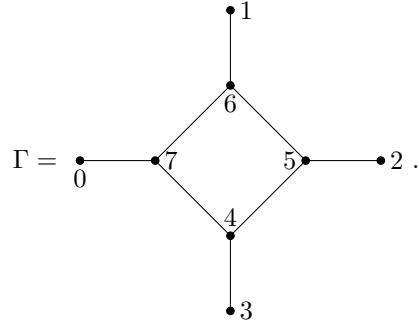
\Gamma=
\begin{tikzpicture}[baseline=-.1cm]
\draw[fill] (0,0) circle (0.05) node[below]{$0$};
\draw (0.,0.) -- (1.,0.);
\draw[fill] (1.,0.) circle (0.05) node[right]{$7$};
\draw (1.,0.) -- (2.,-1);
\draw (1.,0.) -- (2.,1);
\draw[fill] (2.,-1) circle (0.05) node[above]{$4$};
\draw[fill] (2.,1) circle (0.05) node[below]{$6$};
\draw (2.,-1) -- (2.,-2);
\draw (2.,-1) -- (3.,0);
\draw (2.,1) -- (3.,0);
\draw (2.,1) -- (2.,2);
\draw[fill] (2.,-2) circle (0.05) node[right]{$3$};
\draw[fill] (3.,0.) circle (0.05) node[left]{$5$};
\draw[fill] (2.,2) circle (0.05) node[right]{$1$};
\draw (3.,0) -- (4.,0);
\draw[fill] (4.,0) circle (0.05) node [right]{$2$};
\end{tikzpicture}.
$$
Let $\cG_\bullet$ denote the bipartite graph planar algebra of $\Gamma$ \cite{MR1865703}.

Since $\Gamma$ is simply laced, we use a sequence of vertices to express a loop on $\Gamma$.
We express elements $X\in\cG_{n,+}$ as a linear combination of loops of length $2n$.
Thus we use the loop convention of \cite{MR1865703,MR2812459}, not the functional convention of \cite{MR2979509}.

\begin{defn}
Given $X\in\cG_{n,+}$ and a loop $\ell\in \cG_{n,+}$, let $\coeff_{\in X}(\ell)$ be the coefficient of $\ell$ in $X$.
\end{defn}

Recall that the coefficients of loops for a Temperley-Lieb diagram are determined by \cite{MR1865703,MR2812459}.
We use a matrix to express coefficients of multiple loops at the same time.

\begin{defn}
For a vertex $v$ of $\Gamma$, let $p_v\in\cG_{0,\pm}$ be the $0$-box projection corresponding to $v$.
For $v$ an even vertex of $\Gamma$, let $p_v \cG_{n,+}$ be the algebra consisting of elements of the form
$$
\begin{tikzpicture}[baseline=-.1cm]
	\nbox{unshaded}{(-.8,0)}{.25}{0}{0}{$p_v$}
	\draw (0,-.8) -- (0,.8);
	\nbox{unshaded}{(0,0)}{.4}{0}{0}{$X$}
	\node at (.2,-.6) {\scriptsize{$n$}};
	\node at (.2,.6) {\scriptsize{$n$}};
\end{tikzpicture}
$$
where $X\in \cG_{n,+}$ and the multiplication is the usual stacking of elements.

If $n$ is even, and $v,w$ are both even vertices of $\Gamma$, let $p_v \cG_{n,+} p_w$ be the algebra consisting of elements of the form
$$
\begin{tikzpicture}[baseline=-.1cm]
	\nbox{unshaded}{(-.8,0)}{.25}{0}{0}{$p_v$}
	\nbox{unshaded}{(.8,0)}{.25}{0}{0}{$p_w$}
	\draw (0,-.8) -- (0,.8);
	\nbox{unshaded}{(0,0)}{.4}{0}{0}{$X$}
	\node at (.2,-.6) {\scriptsize{$n$}};
	\node at (.2,.6) {\scriptsize{$n$}};
\end{tikzpicture}
$$
where $X\in \cG_{n,+}$ and the multiplication is the usual stacking of elements.
\end{defn}

\begin{remark}
If $v$ is an even vertex, $p_v\cG_{n,+}$ is a direct sum of matrix algebras, and if $w$ is an even vertex and $n$ is even, then $p_v\cG_{n,+}p_w$ is a matrix algebra.
\end{remark}

The following is similar to \cite[Proposition 4.10]{1308.5197}.

\begin{prop}\label{prop:StarHoms}
The following maps are $*$-algebra homomorphisms:
\begin{enumerate}[(1)]
\item
$\Phi_v: \cG_{n,+} \to p_v \cG_{n,+}$ by
$
\begin{tikzpicture}[baseline=-.1cm]
	\draw (0,-.8) -- (0,.8);
	\nbox{unshaded}{(0,0)}{.4}{0}{0}{$X$}
	\node at (.2,-.6) {\scriptsize{$n$}};
	\node at (.2,.6) {\scriptsize{$n$}};
\end{tikzpicture}
\mapsto
\begin{tikzpicture}[baseline=-.1cm]
	\nbox{unshaded}{(-.8,0)}{.25}{0}{0}{$p_v$}
	\draw (0,-.8) -- (0,.8);
	\nbox{unshaded}{(0,0)}{.4}{0}{0}{$X$}
	\node at (.2,-.6) {\scriptsize{$n$}};
	\node at (.2,.6) {\scriptsize{$n$}};
\end{tikzpicture}
$\,, and
\item
$\Phi_{v,w}: \cG_{n,+} \to p_v \cG_{n,+} p_w$ by
$
\begin{tikzpicture}[baseline=-.1cm]
	\draw (0,-.8) -- (0,.8);
	\nbox{unshaded}{(0,0)}{.4}{0}{0}{$X$}
	\node at (.2,-.6) {\scriptsize{$n$}};
	\node at (.2,.6) {\scriptsize{$n$}};
\end{tikzpicture}
\mapsto
\begin{tikzpicture}[baseline=-.1cm]
	\nbox{unshaded}{(-.8,0)}{.25}{0}{0}{$p_v$}
	\nbox{unshaded}{(.8,0)}{.25}{0}{0}{$p_w$}
	\draw (0,-.8) -- (0,.8);
	\nbox{unshaded}{(0,0)}{.4}{0}{0}{$X$}
	\node at (.2,-.6) {\scriptsize{$n$}};
	\node at (.2,.6) {\scriptsize{$n$}};
\end{tikzpicture}
$.
\end{enumerate}
\end{prop}

We now have two subsections to prove Lemmas \ref{lem:OneParameterPlus} and \ref{lem:OneParameterMinus} respectively.
We also check these computations directly in $\cG_\bullet$ in the {\tt Mathematica} notebook {\tt BraidRelationsForS.nb} bundled with the {\tt arXiv} source of this article.

%%%%%%%%%%%%%%%%%%%%%%%%%%%%%%%%%%%%%%%%%%%%%%%%%%%%%%%%%%%%%%%%%%%%%
\subsection{The case \texorpdfstring{$\omega_A=1$}{omega_A = 1}}

We now prove Lemma \ref{lem:OneParameterPlus}.
That is, for $\omega_A=1$, we show that up to sign, there is exactly a 1-parameter family of elements $A=A(\lambda)$ in the graph planar algebra of $\Gamma$ which satisfy
\begin{enumerate}[(1)]
\item
$A^*=A$ and $\cF^2(A)=A$,
\item
$A$ is uncappable, and
\item
$A^2=\jw{2}$ and $\check{A}^2=\cF(A)^2=\jw{2}$.
\end{enumerate}

\begin{proof}[Proof of Lemma \ref{lem:OneParameterPlus}]
We will define the coefficients of the loops for the element $A\in \cG_{2,+}$ and $\check{A}=\cF(A)\in \cG_{2,-}$.
To do so, we will simultaneously find coefficients of projections $P,Q\in \cG_{2,+}$ and $\check{P},\check{Q}\in\cG_{2,-}$ satisfying
\begin{enumerate}
\item
$A=P-Q$ and $\check{A}=\check{P}-\check{Q}$, and
\item
$P+Q=\jw{2}\in\cG_{2,+}$ and $\check{P}+\check{Q}=\jw{2}\in\cG_{2,-}$.
\end{enumerate}
First, note that $p_0 \cG_{2,+}$ is three dimensional. Since $A^2=\jw{2}$, $A$ is uncappable, and
$$
\coeff_{\in e_1}
\begin{pmatrix}
   0707 \\
   0767 \\
   0747
\end{pmatrix}
=
\begin{pmatrix}
   1 \\
   0 \\
   0
\end{pmatrix},
$$
without loss of generality, we may assume that
$$
\coeff_{\in A}
\begin{pmatrix}
   0707 \\
   0767 \\
   0747
\end{pmatrix}
=
\begin{pmatrix}
   0 \\
   1 \\
   -1
\end{pmatrix}.
$$
Then
$$
\coeff_{\in \check{A}}
\begin{pmatrix}
   7070 & 7076 & 7074 \\
   7670 & 7676 & 7674 \\
   7470 & 7476 & 7474
\end{pmatrix}
=
\begin{pmatrix}
   0 & \delta^{-\frac{1}{2}} & -\delta^{-\frac{1}{2}} \\
   \delta^{-\frac{1}{2}} & ? & ? \\
   -\delta^{-\frac{1}{2}} & ? & ?
\end{pmatrix}.
$$
Since $\displaystyle \check{P}=\frac{\jw{2}+\check{A}}{2}$, $\displaystyle \check{Q}=\frac{\jw{2}-\check{A}}{2}$, and $\displaystyle \frac{1-\delta^{-2}}{2}=\delta^{-1}$, we have
\begin{align*}
\coeff_{\in \check{P}}
\begin{pmatrix}
   7070 & 7076 & 7074 \\
   7670 & 7676 & 7674 \\
   7470 & 7476 & 7474
\end{pmatrix}
&=
\begin{pmatrix}
   \delta^{-1} & \frac{-\delta^{-\frac{3}{2}}+\delta^{-\frac{1}{2}}}{2} & \frac{-\delta^{-\frac{3}{2}}-\delta^{-\frac{1}{2}}}{2}\\
   \frac{-\delta^{-\frac{3}{2}}+\delta^{-\frac{1}{2}}}{2} & ? & ? \\
   \frac{-\delta^{-\frac{3}{2}}-\delta^{-\frac{1}{2}}}{2} & ? & ?
\end{pmatrix}
\text{ and }
\\
\coeff_{\in \check{Q}}
\begin{pmatrix}
   7070 & 7076 & 7074 \\
   7670 & 7676 & 7674 \\
   7470 & 7476 & 7474
\end{pmatrix}
&=
\begin{pmatrix}
   \delta^{-1} & \frac{-\delta^{-\frac{3}{2}}-\delta^{-\frac{1}{2}}}{2} & \frac{-\delta^{-\frac{3}{2}}+\delta^{-\frac{1}{2}}}{2} \\
   \frac{-\delta^{-\frac{3}{2}}-\delta^{-\frac{1}{2}}}{2} & ? & ? \\
   \frac{-\delta^{-\frac{3}{2}}+\delta^{-\frac{1}{2}}}{2} & ? & ?
\end{pmatrix}.
\end{align*}
Observe that $p_7 \cG_{2,-} p_7$ is a 3 by 3 matrix algebra.
The images of $e_2,\check{P},\check{Q}\in \cG_{2,-}$ in $p_7 \cG_{2,-} p_7$ under the natural $*$-homomorphism from Proposition \ref{prop:StarHoms} are three mutually orthogonal non-zero projections, so they are all rank 1 projections.
Hence we must have
\begin{align*}
\coeff_{\in \check{P}}
\begin{pmatrix}
   7070 & 7076 & 7074 \\
   7670 & 7676 & 7674 \\
   7470 & 7476 & 7474
\end{pmatrix}
&=
\begin{pmatrix}
   \delta^{-1} & \frac{-\delta^{-\frac{3}{2}}+\delta^{-\frac{1}{2}}}{2} & \frac{-\delta^{-\frac{3}{2}}-\delta^{-\frac{1}{2}}}{2}\\
   \frac{-\delta^{-\frac{3}{2}}+\delta^{-\frac{1}{2}}}{2} & (\frac{-\delta^{-1}+1}{2})^2 & \frac{\delta^{-2}-1}{4} \\
   \frac{-\delta^{-\frac{3}{2}}-\delta^{-\frac{1}{2}}}{2} & \frac{\delta^{-2}-1}{4} & (\frac{\delta^{-1}+1}{2})^2
\end{pmatrix},
\\
\coeff_{\in \check{Q}}
\begin{pmatrix}
   7070 & 7076 & 7074 \\
   7670 & 7676 & 7674 \\
   7470 & 7476 & 7474
\end{pmatrix}
&=
\begin{pmatrix}
   \delta^{-1} & \frac{-\delta^{-\frac{3}{2}}-\delta^{-\frac{1}{2}}}{2} & \frac{-\delta^{-\frac{3}{2}}+\delta^{-\frac{1}{2}}}{2} \\
   \frac{-\delta^{-\frac{3}{2}}-\delta^{-\frac{1}{2}}}{2} & (\frac{\delta^{-1}+1}{2})^2 & \frac{\delta^{-2}-1}{4}  \\
   \frac{-\delta^{-\frac{3}{2}}+\delta^{-\frac{1}{2}}}{2} & \frac{\delta^{-2}-1}{4}  & (\frac{-\delta^{-1}+1}{2})^2
\end{pmatrix},
\text{ and }
\\
\coeff_{\in \check{A}}
\begin{pmatrix}
   7070 & 7076 & 7074 \\
   7670 & 7676 & 7674 \\
   7470 & 7476 & 7474
\end{pmatrix}
&=
\begin{pmatrix}
   0 & \delta^{-\frac{1}{2}} & -\delta^{-\frac{1}{2}} \\
   \delta^{-\frac{1}{2}} & -\delta^{-1} & 0 \\
   -\delta^{-\frac{1}{2}} & 0 & \delta^{-1}
\end{pmatrix}.
\end{align*}
By a similar argument for $p_6 \cG_{2,+}p_4$, which is a 2 by 2 matrix algebra, we have
\begin{align*}
\coeff_{\in A}
\begin{pmatrix}
   6747 & 6745 \\
   6547 & 6545 \\
\end{pmatrix}
&=
\begin{pmatrix}
   0 & ? \\
   ? & ? \\
\end{pmatrix},
\\
\coeff_{\in P}
\begin{pmatrix}
   6747 & 6745 \\
   6547 & 6545 \\
\end{pmatrix}
&=
\begin{pmatrix}
   \frac{1}{2} & ? \\
   ? & ?
\end{pmatrix},
\text{ and}
\\
\coeff_{\in Q}
\begin{pmatrix}
   6747 & 6745 \\
   6547 & 6545
\end{pmatrix}
&=
\begin{pmatrix}
   \frac{1}{2} & ? \\
   ? & ?
\end{pmatrix},
\end{align*}
and the image of $\check{P},\check{Q}$ in $p_6 \cG_{2,+}p_4$ under the natural $*$-homomorphism from Proposition \ref{prop:StarHoms}  are pairwise orthogonal projections. So
\begin{align*}
\coeff_{\in P}
\begin{pmatrix}
   6747 & 6745 \\
   6547 & 6545
\end{pmatrix}
&=
\begin{pmatrix}
   \frac{1}{2} & \frac{1}{2} \lambda \\
   \frac{1}{2} \overline{\lambda} & \frac{1}{2}
\end{pmatrix}
\text{for a phase} ~\lambda,
\\
\coeff_{\in Q}
\begin{pmatrix}
   6747 & 6745 \\
   6547 & 6545
\end{pmatrix}
&=
\begin{pmatrix}
   \frac{1}{2} & -\frac{1}{2} \lambda \\
   -\frac{1}{2} \overline{\lambda} & \frac{1}{2}
\end{pmatrix},
\text{ and}
\\
\coeff_{\in A}
\begin{pmatrix}
   6747 & 6745 \\
   6547 & 6545
\end{pmatrix}
&=
\begin{pmatrix}
   0 & \lambda \\
   \overline{\lambda} & 0
\end{pmatrix}.
\end{align*}
Then
\begin{align*}
\coeff_{\in \check{A}}
\begin{pmatrix}
   7454 & 7456 \\
   7654 & 7656
\end{pmatrix}
&=
\begin{pmatrix}
   ? & \lambda \\
   \overline{\lambda} & ? \\
\end{pmatrix},
\\
\coeff_{\in \check{P}}
\begin{pmatrix}
   7454 & 7456 \\
   7654 & 7656 \\
\end{pmatrix}
&=
\begin{pmatrix}
   ? & \frac{1}{2}\lambda \\
   \frac{1}{2}\overline{\lambda} & ? \\
\end{pmatrix},
\text{ and}
\\
\coeff_{\in \check{Q}}
\begin{pmatrix}
   7454 & 7456 \\
   7654 & 7656 \\
\end{pmatrix}
&=
\begin{pmatrix}
   ? & -\frac{1}{2}\lambda \\
   -\frac{1}{2}\overline{\lambda} & ? \\
\end{pmatrix}.
\end{align*}
There is only one way to realize the two projections in $p_7\cG_{2,-}p_5$ as follows
\begin{align*}
\coeff_{\in \check{P}}
\begin{pmatrix}
   7454 & 7456 \\
   7654 & 7656 \\
\end{pmatrix}
&=
\begin{pmatrix}
   \frac{1}{2} & \frac{1}{2}\lambda \\
   \frac{1}{2}\overline{\lambda} & \frac{1}{2} \\
\end{pmatrix}
\text{ and }
\\
\coeff_{\in \check{Q}}
\begin{pmatrix}
   7454 & 7456 \\
   7654 & 7656 \\
\end{pmatrix}
&=
\begin{pmatrix}
   \frac{1}{2} & -\frac{1}{2}\lambda \\
   -\frac{1}{2}\overline{\lambda} & \frac{1}{2} \\
\end{pmatrix}.
\end{align*}
Then
$$
\coeff_{\in \check{A}}
\begin{pmatrix}
   7454 & 7456 \\
   7654 & 7656 \\
\end{pmatrix}
=
\begin{pmatrix}
   0 & \lambda \\
   \overline{\lambda} & 0
\end{pmatrix}.
$$
By the symmetries of the graph and the symmetries of $A$ and $\check{A}$, we can apply the same argument beginning with $p_1 \cG_{2,-}$, $p_2 \cG_{2,+}$, and $p_3 \cG_{2,-}$, instead of $p_0 \cG_{2,+}$ to derive the coefficients of the other loops.
Therefore we obtain a generator $A=A(\lambda)\in \cG_{2,+}$, which is parameterized by a phase $\lambda$ which satisfies the desired relations by construction.
\end{proof}

%%%%%%%%%%%%%%%%%%%%%%%%%%%%%%%%%%%%%%%%%%%%%%%%%%%%%%%%%%%%%%%%%%%%%
\subsection{The case \texorpdfstring{$\omega_B=-1$}{omega_B = -1}}

We now prove Lemma \ref{lem:OneParameterMinus}.
That is, for $\omega_B=-1$, up to sign, there are exactly two 1-parameter families of elements $B=B(\lambda, \varepsilon)$ in the graph planar algebra of $\Gamma$ which satisfy
\begin{enumerate}[(1)]
\item
$B^*=B$ and $\cF^2(B)=-B$,
\item
$B$ is uncappable, and
\item
$B^2=\jw{2}$ and $\check{B}^2=(-i \cF(B))^2=\jw{2}$.
\end{enumerate}

\begin{proof}[Proof of Lemma \ref{lem:OneParameterMinus}]
We will define the coefficients of the loops for the element $B\in \cG_{2,+}$ and $\check{B}=-i\cF(B)\in \cG_{2,-}$.
To do so, we will simultaneously find coefficients of projections $P,Q\in \cG_{2,+}$ and $\check{P},\check{Q}\in\cG_{2,-}$ satisfying
\begin{enumerate}
\item
$B=P-Q$ and $\check{B}=\check{P}-\check{Q}$, and
\item
$P+Q=\jw{2}\in\cG_{2,+}$ and $\check{P}+\check{Q}=\jw{2}\in\cG_{2,-}$.
\end{enumerate}
As before, without loss of generality, we may assume that
$$
\coeff_{\in B}
\begin{pmatrix}
   0707 \\
   0767 \\
   0747
\end{pmatrix}
=
\begin{pmatrix}
   0 \\
   1 \\
   -1
\end{pmatrix}.
$$
Then
$$
\coeff_{\in \check{B}}
\begin{pmatrix}
   7070 & 7076 & 7074 \\
   7670 & 7676 & 7674 \\
   7470 & 7476 & 7474
\end{pmatrix}
=
\begin{pmatrix}
   0 & i\delta^{-\frac{1}{2}} & -i\delta^{-\frac{1}{2}} \\
   -i\delta^{-\frac{1}{2}} & ? & ? \\
   i\delta^{-\frac{1}{2}} & ? & ?
\end{pmatrix}.
$$
Since $\displaystyle \check{P}=\frac{\jw{2}+\check{B}}{2}$, $\displaystyle \check{Q}=\frac{\jw{2}-\check{B}}{2}$, and $\displaystyle \frac{1-\delta^{-2}}{2}=\delta^{-1}$, we have
\begin{align*}
\coeff_{\in \check{P}}
\begin{pmatrix}
   7070 & 7076 & 7074 \\
   7670 & 7676 & 7674 \\
   7470 & 7476 & 7474
\end{pmatrix}
&=
\begin{pmatrix}
   \delta^{-1} & \frac{-\delta^{-\frac{3}{2}}+i\delta^{-\frac{1}{2}}}{2} & \frac{-\delta^{-\frac{3}{2}}-i\delta^{-\frac{1}{2}}}{2}\\
   \frac{-\delta^{-\frac{3}{2}}-i\delta^{-\frac{1}{2}}}{2} & ? & ? \\
   \frac{-\delta^{-\frac{3}{2}}+i\delta^{-\frac{1}{2}}}{2} & ? & ?
\end{pmatrix}
\text{ and }
\\
\coeff_{\in \check{Q}}
\begin{pmatrix}
   7070 & 7076 & 7074 \\
   7670 & 7676 & 7674 \\
   7470 & 7476 & 7474
\end{pmatrix}
&=
\begin{pmatrix}
   \delta^{-1} & \frac{-\delta^{-\frac{3}{2}}-i\delta^{-\frac{1}{2}}}{2} & \frac{-\delta^{-\frac{3}{2}}+i\delta^{-\frac{1}{2}}}{2} \\
   \frac{-\delta^{-\frac{3}{2}}+i\delta^{-\frac{1}{2}}}{2} & ? & ? \\
   \frac{-\delta^{-\frac{3}{2}}-i\delta^{-\frac{1}{2}}}{2} & ? & ?
\end{pmatrix}.
\end{align*}
Observe that $p_7 \cG_{2,-} p_7$ is a 3 by 3 matrix algebra.
The images of $e_2,\check{P},\check{Q}\in \cG_{2,-}$ in $p_7 \cG_{2,-} p_7$ under the natural $*$-homomorphism from Proposition \ref{prop:StarHoms} are three mutually orthogonal non-zero projections, so they are all rank 1 projections.
Hence we must have
\begin{align*}
\coeff_{\in \check{P}}
\begin{pmatrix}
   7070 & 7076 & 7074 \\
   7670 & 7676 & 7674 \\
   7470 & 7476 & 7474
\end{pmatrix}
&=
\begin{pmatrix}
   \delta^{-1} & \frac{-\delta^{-\frac{3}{2}}+i\delta^{-\frac{1}{2}}}{2} & \frac{-\delta^{-\frac{3}{2}}-i\delta^{-\frac{1}{2}}}{2}\\
   \frac{-\delta^{-\frac{3}{2}}-i\delta^{-\frac{1}{2}}}{2} & \frac{1+\delta^{-2}}{4} & (\frac{-\delta^{-1}-i}{2})^2 \\
   \frac{-\delta^{-\frac{3}{2}}+i\delta^{-\frac{1}{2}}}{2} & (\frac{-\delta^{-1}+i}{2})^2  & \frac{1+\delta^{-2}}{4}
\end{pmatrix},
\\
\coeff_{\in \check{Q}}
\begin{pmatrix}
   7070 & 7076 & 7074 \\
   7670 & 7676 & 7674 \\
   7470 & 7476 & 7474
\end{pmatrix}
&=
\begin{pmatrix}
   \delta^{-1} & \frac{-\delta^{-\frac{3}{2}}-i\delta^{-\frac{1}{2}}}{2} & \frac{-\delta^{-\frac{3}{2}}+i\delta^{-\frac{1}{2}}}{2} \\
   \frac{-\delta^{-\frac{3}{2}}+i\delta^{-\frac{1}{2}}}{2} & \frac{1+\delta^{-2}}{4} & (\frac{-\delta^{-1}+i}{2})^2  \\
   \frac{-\delta^{-\frac{3}{2}}-i\delta^{-\frac{1}{2}}}{2} & (\frac{-\delta^{-1}-i}{2})^2  & \frac{1+\delta^{-2}}{4}
\end{pmatrix},
\text{ and }
\\
\coeff_{\in \check{B}}
\begin{pmatrix}
   7070 & 7076 & 7074 \\
   7670 & 7676 & 7674 \\
   7470 & 7476 & 7474
\end{pmatrix}
&=
\begin{pmatrix}
   0 & i\delta^{-\frac{1}{2}} & -i\delta^{-\frac{1}{2}} \\
   -i\delta^{-\frac{1}{2}} & 0 & i\delta^{-1} \\
   i\delta^{-\frac{1}{2}} & -i\delta^{-1} & 0
\end{pmatrix}
\end{align*}
By a similar argument for $p_6 \cG_{2,+}p_4$, which is a 2 by 2 matrix algebra, we have
\begin{align*}
\coeff_{\in B}
\begin{pmatrix}
   6747 & 6745 \\
   6547 & 6545 \\
\end{pmatrix}
&=
\begin{pmatrix}
   \delta^{-1} & ? \\
   ? & ? \\
\end{pmatrix},
\\
\coeff_{\in P}
\begin{pmatrix}
   6747 & 6745 \\
   6547 & 6545 \\
\end{pmatrix}
&=
\begin{pmatrix}
   \frac{1+\delta^{-1}}{2} & ? \\
   ? & ?
\end{pmatrix},
\text{ and}
\\
\coeff_{\in Q}
\begin{pmatrix}
   6747 & 6745 \\
   6547 & 6545
\end{pmatrix}
&=
\begin{pmatrix}
   \frac{1-\delta^{-1}}{2} & ? \\
   ? & ?
\end{pmatrix},
\end{align*}
and the image of $\check{P},\check{Q}$ in $p_6 \cG_{2,+}p_4$ under the natural $*$-homomorphism from Proposition \ref{prop:StarHoms}  are pairwise orthogonal projections. So
\begin{align*}
\coeff_{\in P}
\begin{pmatrix}
   6747 & 6745 \\
   6547 & 6545
\end{pmatrix}
&=
\begin{pmatrix}
   \frac{1+\delta^{-1}}{2} & \lambda\frac{\sqrt{1-\delta^{-2}}}{2}  \\
  \overline{\lambda} \frac{\sqrt{1-\delta^{-2}}}{2}  & \frac{1-\delta^{-1}}{2}
\end{pmatrix}
\text{for a phase} ~\lambda,
\\
\coeff_{\in Q}
\begin{pmatrix}
   6747 & 6745 \\
   6547 & 6545
\end{pmatrix}
&=
\begin{pmatrix}
   \frac{1-\delta^{-1}}{2} & -\lambda\frac{\sqrt{1-\delta^{-2}}}{2}  \\
   -\overline{\lambda}\frac{\sqrt{1-\delta^{-2}}}{2}  & \frac{1+\delta^{-1}}{2}
\end{pmatrix},
\text{ and}
\\
\coeff_{\in B}
\begin{pmatrix}
   6747 & 6745 \\
   6547 & 6545
\end{pmatrix}
&=
\begin{pmatrix}
   \delta^{-1} &\lambda \sqrt{1-\delta^{-2}}  \\
  \overline{\lambda}  \sqrt{1-\delta^{-2}} & -\delta^{-1}
\end{pmatrix}.
\end{align*}
Then
\begin{align*}
\coeff_{\in \check{B}}
\begin{pmatrix}
   7454 & 7456 \\
   7654 & 7656
\end{pmatrix}
&=
\begin{pmatrix}
   ? & -i\lambda \sqrt{1-\delta^{-2}}  \\
   i\overline{\lambda}\sqrt{1-\delta^{-2}}  & ?
\end{pmatrix},
\\
\coeff_{\in \check{P}}
\begin{pmatrix}
   7454 & 7456 \\
   7654 & 7656 \\
\end{pmatrix}
&=
\begin{pmatrix}
   ? & -\frac{i}{2}\lambda \sqrt{1-\delta^{-2}}  \\
   \frac{i}{2}\overline{\lambda}\sqrt{1-\delta^{-2}} & ? \\
\end{pmatrix},
\text{ and}
\\
\coeff_{\in \check{Q}}
\begin{pmatrix}
   7454 & 7456 \\
   7654 & 7656 \\
\end{pmatrix}
&=
\begin{pmatrix}
   ? & \frac{i}{2} \lambda\sqrt{1-\delta^{-2}}  \\
   -\frac{i}{2}\overline{\lambda}\sqrt{1-\delta^{-2}}  & ? \\
\end{pmatrix}.
\end{align*}
There are two ways to realize the two projections in $p_7\cG_{2,-}p_5$ based on the sign $\varepsilon$ as follows
\begin{align*}
\coeff_{\in \check{P}}
\begin{pmatrix}
   7454 & 7456 \\
   7654 & 7656 \\
\end{pmatrix}
&=
\begin{pmatrix}
   \frac{1-\varepsilon\delta^{-1}}{2} & -\frac{i}{2} \lambda\sqrt{1-\delta^{-2}}  \\
   \frac{i}{2}\overline{\lambda}\sqrt{1-\delta^{-2}} & \frac{1+\varepsilon\delta^{-1}}{2} \\
\end{pmatrix}
\text{ and }
\\
\coeff_{\in \check{Q}}
\begin{pmatrix}
   7454 & 7456 \\
   7654 & 7656 \\
\end{pmatrix}
&=
\begin{pmatrix}
   \frac{1+\varepsilon\delta^{-1}}{2} & \frac{i}{2} \lambda\sqrt{1-\delta^{-2}}  \\
   -\frac{i}{2}\overline{\lambda}\sqrt{1-\delta^{-2}} & \frac{1-\varepsilon\delta^{-1}}{2} \\
\end{pmatrix}.
\end{align*}
Then
$$
\coeff_{\in \check{B}}
\begin{pmatrix}
   7454 & 7456 \\
   7654 & 7656 \\
\end{pmatrix}
=
\begin{pmatrix}
   -\varepsilon\delta^{-1} & -i \lambda\sqrt{1-\delta^{-2}}  \\
   i \overline{\lambda}\sqrt{1-\delta^{-2}}  & \varepsilon\delta^{-1} \\
\end{pmatrix}.
$$

By the symmetries of the graph and the symmetries of $B$ and $\check{B}$, we can apply the same argument beginning with $p_1 \cG_{2,-}$, $p_2 \cG_{2,+}$, and $p_3 \cG_{2,-}$, instead of $p_0 \cG_{2,+}$ to derive the coefficients of the other loops.
Therefore we obtain a generator $B=B(\lambda,\varepsilon)\in \cG_{2,+}$, which is parameterized by a phase $\lambda$ and a sign $\varepsilon$ which satisfies the desired relations by construction.
\end{proof}

\bibliographystyle{alpha}
\bibliography{../../bibliography/bibliography}

\end{document}